\documentclass[a4paper,reqno]{amsart}

\theoremstyle{plain}

\newtheorem{thm}{Theorem}[section]
\newtheorem{lem}[thm]{Lemma}
\newtheorem{cor}[thm]{Corollary}
\newtheorem{rem}[thm]{Remark}

\usepackage{amsmath}
\usepackage{nccmath}
\usepackage{amssymb}
\usepackage{bm}
\usepackage{mathrsfs}
\usepackage{type1cm}
\usepackage[ppl]{mathcomp}
\usepackage{graphicx}
\usepackage{amsthm}
\usepackage{comment}
\usepackage{mathtools}
\allowdisplaybreaks

\makeatletter
    
    \@addtoreset{equation}{section}
\makeatother

\def\doubleprod#1#2{\ooalign{$#1\prod$\cr$#1\coprod$\cr}}
\DeclareMathOperator*{\Rprod}{\mathpalette\doubleprod\relax}

\def\dsp{ \displaystyle }

\def\rhoco{ \rho_{ \chi_1 } }
\def\rhoct{ \rho_{ \chi_2 } }
\def\rhoca{ \rho_{ \chi_a } }
\def\rhocb{ \rho_{ \chi_b } }
\def\rhocj{ \rho_{ \chi_j } }
\def\rhocobar{ \rho_{ \bar{ \chi_1 } } }
\def\rhoctbar{ \rho_{ \bar{ \chi_2 } } }
\def\rhocabar{ \rho_{ \bar{ \chi_a } } }
\def\rhocbbar{ \rho_{ \bar{ \chi_b } } }

\def\tauco{ \tau_{ \chi_1 } }

\def\taucj{ \tau_{ \chi_j } }
\def\taucobar{ \tau_{ \bar{ \chi_1 } } }

\def\taucjbar{ \tau_{ \bar{ \chi_j } } }

\def\taucjm{ \tau_{ \chi_j }^{ ( 1 ) } }

\def\taucjbarm{ \tau_{ \bar{ \chi_j } }^{ ( 1 ) } }
\def\taucoz{ \tau_{ \chi_1 }^{ (0) } }
\def\tauctz{ \tau_{ \chi_2 }^{ (0) } }
\def\taucaz{ \tau_{ \chi_a }^{ ( 0 ) } }
\def\taucbz{ \tau_{ \chi_b }^{ ( 0 ) } }
\def\taucjz{ \tau_{ \chi_j }^{ ( 0 ) } }
\def\tauoz{ \tau_1^{ ( 0 ) } }
\def\tautz{ \tau_2^{ ( 0 ) } }
\def\taurz{ \tau_r^{ ( 0 ) } }
\def\muoz{ \mu_{ \chi_1 } ( 0 ) }
\def\mutz{ \mu_{ \chi_2 } ( 0 ) }
\def\muaz{ \mu_{ \chi_a } ( 0 ) }
\def\mubz{ \mu_{ \chi_b } ( 0 ) }
\def\mujz{ \mu_{ \chi_j } ( 0 ) }
\def\mutauoz{ \mu_{ \chi_1 } ( \tau_{ \chi_1 }^{ ( 0 ) } ) }
\def\mutautz{ \mu_{ \chi_2 } ( \tau_{ \chi_2 }^{ ( 0 ) } ) }
\def\mutauaz{ \mu_{ \chi_a } ( \tau_{ \chi_a }^{ ( 0 ) } ) }
\def\mutaubz{ \mu_{ \chi_b } ( \tau_{ \chi_b }^{ ( 0 ) } ) }
\def\mutaujz{ \mu_{ \chi_j } ( \tau_{ \chi_j }^{ ( 0 ) } ) }

\begin{document}
\title[Euler prod. of Absolute tensor prod.]{The Euler product expressions of the absolute tensor products of the Dirichlet $ L $-functions}
\author[H.~Tanaka]{Hidenori Tanaka}
\date{}
\subjclass[2010]{Primary~11M06}
\keywords{Dirichlet $ L $-function; Absolute tensor product(Kurokawa tensor product); Euler product}
\begin{abstract}
\noindent
In this paper, we calculate the absolute tensor square of the Dirichlet $ L $-functions and show that it is expressed as an Euler product over pairs of primes. The method is to construct an equation to link primes to a series which has the factors of the absolute tensor product of the Dirichlet $ L $-functions. This study is a generalization of Akatsuka's theorem on the Riemann zeta function, and gives a proof of Kurokawa's prediction proposed in 1992.
\end{abstract}
\maketitle
\section{\textbf{Introduction}} \label{Sec1}
In 1992 Kurokawa \cite{Kur} defined the absolute tensor products (Kurokawa tensor products). The definition is given by
\begin{align*}
( Z_1 \underset{ \mathbb{ F }_1 }{ \otimes } \cdots \underset{ \mathbb{ F }_1 }{ \otimes } Z_r ) ( s ) \coloneqq \Rprod_{ \rho_1 , \cdots , \rho_r \in \mathbb{ C } } (( s - \rho_1 - \cdots - \rho_r ))^{ \mu ( \rho_1 , \cdots , \rho_r ) }
\end{align*}
for some zeta functions $ Z_j ( s ) \, ( j = 1 , \cdots , r ) $, where the symbol $ \Rprod $, which was introduced by Deninger \cite{Den}, represents the zeta regularized product (see below) and the integer $ \mu ( \rho_1 , \cdots , \rho_r ) $ is defined by
\begin{align*}
\mu ( \rho_1 , \cdots , \rho_r ) \coloneqq \mu_1 ( \rho_1 ) \cdots \mu_r ( \rho_r ) \times \begin{cases}
1 & ( \Im ( \rho_1 ) , \cdots , \Im ( \rho_r ) \ge 0 ), \\
( - 1 )^{ r - 1 } & ( \Im ( \rho_1 ) , \cdots , \Im ( \rho_r ) < 0 ), \\
0 & ( \mbox{otherwise} ),
\end{cases}
\end{align*}
where $ \mu_j ( \rho ) $ denotes the order of $ \rho $ which is a zero of $ Z_j ( s ) $; now, we regard the poles of $ Z_j ( s ) $ as the zeros with negative orders in this paper. Here the zeta regularized products are defined by
\begin{align*}
\Rprod_{ n = 1 }^\infty (( s - a_n ))^{ b_n } \coloneqq \exp \left( - \underset{ w = 0 }{ \mathrm{Res} } \frac{ Z_{ \bm{ a }, \bm{ b } } ( w , s ) }{ w^2 } \right)
\end{align*}
where $ \bm{ a } \coloneqq \{ a_n \}_{ n = 1 }^\infty $ and $ \bm{ b } \coloneqq \{ b_n \}_{ n = 1 }^\infty $ are complex sequences such that $ Z_{ \bm{ a } , \bm{ b } } ( w , s ) \coloneqq \sum_{ n = 1 }^\infty b_n ( s - a_n )^{ - w } $ converges locally, uniformly and absolutely in some $ s $-region included in $ \mathbb{ C } - \bm{ a } $ for $ \Re ( w ) > C $ with some constant $ C \in \mathbb{ R }_{ > 0 } $ and is a meromorphic function of $ w $ at $ w = 0 $. If $ \bm{ b } \subset \mathbb{ Z } $ then $ \Rprod_{ n = 1 }^\infty (( s - a_n ))^{ b_n } $ is a meromorphic function of $ s $ in the whole $ \mathbb{ C } $ and has zeros only at $ s = a_n $. The integer $ b_n $ contributes to the order of $ a_n $. See \cite{KurWak1} for more details concerning the zeta regularized products. The factors of the zeta regularized products are derived from the summands of $ Z_{ \bm{ a } , \bm{ b } } ( w , s ) $, so we call $ Z_{ \bm{ a } , \bm{ b } } ( w , s ) $ the ``factors series'' in this paper.

In \cite{Kur} Kurokawa also predicted that the absolute tensor product of $ r $ arithmetic zeta functions which have the expression by the Euler product over primes would have the Euler product over $ r $-tuples $ ( p_1 , \cdots , p_r ) $ of primes. The validity of Kurokawa's prediction has been confirmed in some cases, for example, the cases of the Hasse zeta functions of finite fields by Koyama and Kurokawa \cite{KoyKur} for $ r = 2 $, by Akatsuka \cite{Aka1} for $ r = 3 $ and by Kurokawa and Wakayama \cite{KurWak2} for general $ r $. Also, the case of the Riemann zeta function for $ r = 2 $ was first proved by Koyama and Kurokawa \cite{KoyKur}, and then by Akatsuka \cite{Aka2} in a different way.

In this paper, according to Akatsuka's method in \cite{Aka2}, we will reach the Euler product expression of the absolute tensor product $ (L_{ \chi_1 } \underset{ \mathbb{ F }_1 }{ \otimes } L_{ \chi_2 } ) ( s ) $, where $ L_{ \chi_j } (s) \coloneqq L ( s , \chi_j ) \, ( j \in \mathbb{ Z }_{ >0 } ) $ denotes the Dirichlet $ L $-function corresponding to a primitive Dirichlet character $ \chi_j $ to the modulus $ N_j $ with $ N_j \in \mathbb{ Z }_{ \ge 2 } $. The key item which leads to our goal is an equation which links the ``factors series'' of $ ( L_{ \chi_1 } \underset{ \mathbb{ F }_1 }{ \otimes } \cdots \underset{ \mathbb{ F }_1 }{ \otimes } L_{ \chi_r } ) ( s ) $ to $ r $-tuples of prime numbers (see Theorem 4.1 below). We name such equation the ``key equation''. Letting $ r = 1 $, where $ r $ is a parameter in the ``key equation'', we obtain the zeta regularized product expression of $ L ( s , \chi_1 ) $ :
\begin{thm} \label{Thm1.1}
Let $ \rhoco $ denote the imaginary zeros of $ L ( s , \chi_1 ) $ counted with multiplicity, and let $ \taucoz $ be a possible real number with $ 0 < \taucoz < \frac{ 1 }{ 2 } $ and $ L \left( \frac{ 1 }{ 2 } \pm \taucoz , \chi_1 \right) = 0 $. Then $ L ( s , \chi_1 ) $ has the following expression :
\begin{align}
\begin{split}
L( s , \chi_1 ) = & \Rprod_{ \Im ( \rhoco ) \not = 0 } (( s - \rhoco )) \Rprod_{ n = 1 }^\infty \left( \left( s + 2n - \frac{ 3 + \chi_1 ( -1 ) }{ 2 } \right) \right) \\
&\times \left( s - \frac{ 1 }{ 2 } - \taucoz \right)^{ \mutauoz } \left( s - \frac{ 1 }{ 2 } + \taucoz \right)^{ \mutauoz } \left( s - \frac{ 1 }{ 2 } \right)^{ \muoz },
\end{split} \label{Thm1.1-1}
\end{align}
where $ \mutauoz $ and $ \muoz $ denote the order of $ \frac{ 1 }{ 2 } \pm \taucoz $ and the order of $ \frac{ 1 }{ 2 } $ respectively.
\end{thm}
\begin{rem}
As Theorem \ref{Thm1.1}, define $ \taucjz $ by $ 0 < \taucjz < \frac{ 1 }{ 2 } $ and $ L \left( \frac{ 1 }{ 2 } \pm \taucjz , \chi_j \right) = 0 $ for $ j \in \mathbb{ Z }_{ > 0 } $. It is well known that the orders of $ \frac{ 1 }{ 2 } \pm \taucjz $ are equal and at most one. For convenience, if $ \mutaujz = 0 $ then we define $ \taucjz \coloneqq \frac{ 1 }{ 4 } $.
\end{rem}
Let $ \rho_{ \chi_j } $ denote the imaginary zeros of $ L ( s , \chi_j ) $.  From $ ( \ref{Thm1.1-1} ) $ and the definition of the absolute tensor products, we find that $ ( L_{ \chi_1 } \underset{ \mathbb{ F }_1 }{ \otimes } L_{ \chi_2 } ) ( s ) $ has the following expression :
\begin{align*}
& ( L_{ \chi_1 } \underset{ \mathbb{ F }_1 }{ \otimes } L_{ \chi_2 } ) ( s ) \\
& = \Rprod_{ \Im ( \rhoco ) , \Im ( \rhoct ) < 0 } (( s - \rhoco - \rhoct ))^{ - 1 } \Rprod_{ \Im ( \rhoco ) , \Im ( \rhoct ) > 0 } (( s - \rhoco - \rhoct )) \\
& \quad \times \prod_{ ( a , b ) \in \{ ( 1 , 2 ) , ( 2 , 1 ) \} } \left( \Rprod_{ \Im ( \rhoca ) > 0 , n \ge 1 } \left( \left( s - \rhoca + 2n - \frac{ 3 + \chi_b ( - 1 ) }{ 2 } \right) \right) \right. \\
& \phantom{ \quad \quad \times \prod_{ ( a , b ) \in \{ ( 1 , 2 ) , ( 2 , 1 ) \} } } \times \Rprod_{ \Im ( \rhoca ) > 0 } \left( \left( s - \rhoca - \frac{ 1 }{ 2 } - \taucbz \right) \right)^{ \mutaubz } \\
& \phantom{ \quad \quad \times \prod_{ ( a , b ) \in \{ ( 1 , 2 ) , ( 2 , 1 ) \} } } \times \Rprod_{ \Im ( \rhoca ) > 0 } \left( \left( s - \rhoca - \frac{ 1 }{ 2 } + \taucbz \right) \right)^{ \mutaubz } \\
& \phantom{ \quad \quad \times \prod_{ ( a , b ) \in \{ ( 1 , 2 ) , ( 2 , 1 ) \} } } \times \Rprod_{ \Im ( \rhoca ) > 0 } \left( \left( s - \rhoca - \frac{ 1 }{ 2 } \right) \right)^{ \mubz } \\
& \phantom{ \quad \quad \times \prod_{ ( a , b ) \in \{ ( 1 , 2 ) , ( 2 , 1 ) \} } } \times \Rprod_{ n \ge 1 } \left( \left( s + 2n - 2 - \frac{ \chi_a ( - 1 ) }{ 2 } - \taucbz \right) \right)^{ \mutaubz } \\
& \phantom{ \quad \quad \times \prod_{ ( a , b ) \in \{ ( 1 , 2 ) , ( 2 , 1 ) \} } } \times \Rprod_{ n \ge 1 } \left( \left( s + 2n - 2 - \frac{ \chi_a ( - 1 ) }{ 2 } + \taucbz \right) \right)^{ \mutaubz } \\
& \phantom{ \quad \quad \times \prod_{ ( a , b ) \in \{ ( 1 , 2 ) , ( 2 , 1 ) \} } } \times \left. \Rprod_{ n \ge 1 } \left( \left( s + 2n - 2 - \frac{ \chi_a ( - 1 ) }{ 2 } \right) \right)^{ \mubz } \right) \\
& \quad \times \Rprod_{ n_1 , n_2 \ge 1 } \left( \left( s + 2 n_1 + 2 n_2 - 3 - \frac{ \chi_1 ( -1 ) + \chi_2 (-1) }{ 2 } \right) \right) \\
& \quad \times \left( ( s - 1 - \taucoz - \tauctz ) ( s - 1 - \taucoz + \tauctz ) \right. \\
& \quad \quad \times \left. ( s - 1 + \taucoz - \tauctz ) ( s - 1 + \taucoz + \tauctz ) \right)^{ \mutauoz \mutautz } \\
& \quad \times \prod_{ ( a , b ) \in \{ ( 1 , 2 ) , ( 2 , 1 ) \} } \left( ( s - 1 - \taucaz )( s - 1 + \taucaz ) \right)^{ \mutauaz \mubz } \\
& \quad \times ( s - 1 )^{ \muoz \mutz }.
\end{align*}
Now, let $ p, \, q $ be primes and $ j, \, m, \, n $ be positive integers, and let $ \alpha $ be any fixed number with $ 0 < \alpha < 1 $. For the complex numbers $ \taucj $ with $ \rhocj = \frac{ 1 }{ 2 } + i \taucj $, we define $ \taucjm \coloneqq \min \{ \Re ( \taucj ) > 0 \} $; we fix $ \varepsilon_j $ arbitrarily with $ 0 < \varepsilon_j < \min \{ \taucjm , \taucjbarm \} $. Also, we define $ \varepsilon^{ ( r ) } \coloneqq \underset{ j \in \{ 1 , \cdots , r \} }{ \min } \{ \varepsilon_j \} $. Define that
\begin{align*}
&E_1 ( w , s , \{ \chi_j \}_{ j =1 }^2 ) \\
& \coloneqq - \frac{ i ( w + 1 ) }{ 2 \pi } \sum_{ p } \sum_{ m = 1 }^\infty \chi_1 ( p^m ) \chi_2 ( p^m ) p^{ - ms } ( m \log p )^{ w - 2 } ( \log p )^2 \\
& \quad + \frac{ i ( s - 2 ) }{ 2 \pi } \sum_{ p } \sum_{ m = 1 }^\infty \chi_1 ( p^m ) \chi_2 ( p^m ) p^{ - ms } ( m \log p )^{ w - 1 } ( \log p )^2, \\%
&E_2 ( w , s , \{ \chi_j \}_{ j = 1 }^2 ) \\
& \coloneqq - \frac{ i }{ 2 \pi } \sum_{ ( a , b ) \in \{ ( 1 , 2 ) , ( 2 , 1 ) \} } \sum_{ \scriptsize{ \begin{array}{c}
p , m , q , n \\
p^m \not = q^n
\end{array} } } \frac{ \chi_a ( p^m ) \chi_b ( q^n ) p^{ -m ( s - 1 ) } q^{ - n } ( m \log p )^w \log p }{ n ( m \log p - n \log q ) }, \\
&E_3 ( w , s , \{ \chi_j \}_{ j = 1 }^2 ) \\
& \coloneqq \frac{ 1 }{ 2 \pi } \sum_{ ( a , b ) \in \{ ( 1 , 2 ) , ( 2 , 1 ) \} } \sum_{ p , m , q , n } \frac{ \chi_a ( p^m ) \bar{ \chi_b } ( q^n ) p^{ - m ( s + \alpha ) } q^{ - n ( 1 + \alpha ) } }{ n ( m \log p + n \log q ) } ( m \log p )^w \log p, \\
&E_4 ( w , s , \{ \chi_j \}_{ j=1 }^2 ) \\
& \coloneqq - \frac{ 1 }{ 2 \pi } \sum_{ ( a , b ) \in \{ ( 1 , 2 ) , ( 2 , 1 ) \} } \sum_{ p , m , n } \frac{ \chi_a ( - 1 ) \chi_b ( p^m ) p^{ - m ( s + \alpha ) } q^{ - \alpha n \pi i } }{ n ( im \log p - n \pi ) } ( m \log p )^{ w - 2 } \log p, \\
&E_5 ( w , s , \{ \chi_j \}_{ j = 1 }^2 ) \\
& \coloneqq \frac{ i }{ 2 } \sum_{ ( a , b ) \in \{ ( 1 , 2 ) , ( 2 , 1 ) \} } \sum_{ p , m } \frac{ \chi_a ( p^m ) p^{ -m \left( s - \frac{ 1 - \chi_b ( -1 ) }{ 2 } \right) } }{ \sin ( im \log p ) } ( m \log p )^{ w - 1 } \log p, \\
&E_6 ( w , s , \{ \chi_j \}_{ j = 1 }^2 ) \\
& \coloneqq - \frac{ i }{ 2 \pi } \sum_{ ( a , b ) \in \{ ( 1 , 2 ) , ( 2 , 1 ) \} } \int_{ S^{ ( 2 ) } ( \pi \to 0 ) } \sum_{ p , m } \chi_a ( p^m ) p^{ - m ( s - u ) } ( m \log p )^w ( \log p ) \\
& \phantom{ \coloneqq - \frac{ i }{ 2 \pi } \sum_{ ( a , b ) \in \{ ( 1 , 2 ) , ( 2 , 1 ) \} } \int_{ S^{ ( 2 ) } ( \pi \to 0 ) } \sum_{ p , m } \chi_a ( p^m ) p^{ - m ( s - u ) } \qquad } \times \log L ( u , \chi_b ) du, \\
&E_7 ( w , s , \{ \chi_j \}_{ j = 1 }^2 ) \\
& \coloneqq \sum_{ ( a , b ) \in \{ ( 1 , 2 ) , ( 2 , 1 ) \} } \frac{ 1 }{ 2 \pi } \left( \log \left( \frac{ \chi_a ( -1 ) \Gamma ( 1 + \alpha ) N_a^\alpha G( \chi_a ) }{ ( 2 \pi )^{ 1 + \alpha } } \right) + \gamma + \log \left( \frac{ 2 \pi }{ N_a } + \frac{ \pi i }{ 2 } \right) \right) \\
& \phantom{ \log \left( \frac{ 2 \pi }{ N_a } + \frac{ \pi i }{ 2 } \right) } \times \sum_{ p , m } \chi_b ( p^m ) p^{ -m ( s + \alpha ) } ( m \log p )^{ w - 2 } \log p \\
& \quad + \sum_{ a = 1 }^2 \left( - \frac{ 1 + \alpha }{ 4 } \sum_{ p , m } \chi_a ( p^m ) p^{ - m ( s + \alpha ) } ( m \log p )^{ w - 1 } \log p \right. \\
& \quad \phantom{ - \frac{ 1 + \alpha }{ 4 } } + \frac{ i }{ 2 \pi } \sum_{ p , m } \chi_a ( p^m ) p^{ - m ( s + \alpha ) } ( m \log p )^{ w - 1 } \log p \\
& \quad \left. \phantom{ - \frac{ 1 + \alpha }{ 4 } } \times \int_0^\infty \frac{ 1 }{ e^u - 1 } \cdot \frac{ u + m ( \log p ) ( 1 - e^{ - \alpha u } ) }{ u + m \log p } du \right), \\
&E_8 ( w , s , \{ \chi_j \}_{ j = 1 }^2 ) \\
& \coloneqq \sum_{ ( a , b ) \in \{ ( 1 , 2 ) , ( 2 , 1 ) \} } \mutauaz \sum_{ p , m } \chi_b ( p^m ) p^{ - m \left( s - \frac{ 1 }{ 2 } - \taucaz \right) } ( m \log p )^{ w - 1 } \log p, \\
&E_9 ( w , s , \{ \chi_j \}_{ j = 1 }^2 ) \\
& \coloneqq \sum_{ ( a , b ) \in \{ ( 1 , 2 ) , ( 2 , 1 ) \} } \mutauaz \sum_{ p , m } \chi_b ( p^m ) p^{ - m \left( s - \frac{ 1 }{ 2 } + \taucaz \right) } ( m \log p )^{ w - 1 } \log p, \\
&E_{ 10 } ( w , s , \{ \chi_j \}_{ j = 1 }^2 ) \\
& \coloneqq \sum_{ ( a , b ) \in \{ ( 1 , 2 ) , ( 2 , 1 ) \} } \muaz \sum_{ p , m } \chi_b ( p^m ) p^{ - m \left( s - \frac{ 1 }{ 2 } \right) } ( m \log p )^{ w - 1 } \log p,
\end{align*}
where $ S^{ ( r ) } : = \left\{ \frac{ 1 + \alpha }{ 2 } \cos \varphi + i \varepsilon^{ ( r ) } \sin \varphi + \frac{ 1 - \alpha }{ 2 } \mid 0 \le \varphi \le \pi \right\} $ and $ \gamma $, $ \Gamma ( s ) $ and $ G ( \chi_j ) $ denote the Euler constant, the gamma function and the Gauss sum respectively, that is,
\begin{align*}
& \gamma \coloneqq \lim_{ K \to \infty } \left( \sum_{ k = 1 }^K \frac{ 1 }{ k } - \log K \right), \\
& \Gamma ( s ) \coloneqq \dsp{ \int_0^\infty e^{ - t } t^{ s - 1 } dt } \quad ( \Re ( s ) > 0 ), \\
& G ( \chi_j ) \coloneqq \sum_{ n = 1 }^{ N_j } \chi_j ( n ) e^{ \frac{ 2 \pi i }{ N_j } n }.
\end{align*}
Then, letting  $ r = 2 $ in the ``key equation'', we can deduce the Euler product expression of $ ( L_{ \chi_1 } \underset{ \mathbb{ F }_1 }{ \otimes } L_{ \chi_2 } ) ( s ) $ as follows :
\begin{thm} \label{Thm1.3}
In $ \Re ( s ) > 2 $ we have
\begin{align*}
( L_{ \chi_1 } \underset{ \mathbb{ F }_1 }{ \otimes } L_{ \chi_2 } ) ( s ) = \exp \left( \sum_{ k = 1 }^{ 10 } E_k ( s , \{ \chi_j \}_{ j = 1 }^2 ) \right) ,
\end{align*}
where $ E_k ( s , \{ \chi_j \}_{ j = 1 }^2 ) \coloneqq E_k ( 0 , s , \{ \chi_j \}_{ j = 1 }^2 ) $, that is,
\begin{align*}
&E_1 ( s , \{ \chi_j \}_{ j = 1 }^2 ) \\
& \coloneqq - \frac{ i }{ 2 \pi } \sum_{ p , m } \frac{ \chi_1 ( p^m ) \chi_2 ( p^m ) p^{ - ms } }{ m^2 } + \frac{ i ( s - 2 ) }{ 2 \pi } \sum_{ p , m } \frac{ \chi_1 ( p^m ) \chi_2 ( p^m ) p^{ - ms } \log p }{ m }, \\
&E_2 ( s , \{ \chi_j \}_{ j = 1 }^2 ) \\
& \coloneqq - \frac{ i }{ 2 \pi } \sum_{ ( a , b ) \in \{ ( 1 , 2 ) , ( 2 , 1 ) \} } \sum_{ \scriptsize{ \begin{array}{c}
p , m , q , n \\
p^m \not = q^n
\end{array} } } \frac{ \chi_a ( p^m ) \chi_b ( q^n ) p^{ - m ( s - 1 ) } q^{ -n } \log p }{ n ( m \log p - n \log q ) }, \\
&E_3 ( s , \{ \chi_j \}_{ j = 1 }^2 ) \\
& \coloneqq \frac{ 1 }{ 2 \pi } \sum_{ ( a , b ) \in \{ ( 1 , 2 ) , ( 2 , 1 ) \} } \sum_{ p , m , q , n } \frac{ \chi_a ( p^m ) \chi_b ( q^n ) p^{ - m ( s + \alpha )} q^{ - n ( 1 + \alpha ) } }{ n ( m \log p + n \log q ) } \log q, \\
&E_4 ( s , \{ \chi_j \}_{ j = 1 }^2 ) \\
& \coloneqq - \frac{ 1 }{ 2 \pi } \sum_{ ( a , b ) \in \{ ( 1 , 2 ) , ( 2 , 1 ) \} } \sum_{ p , m , n } \frac{ \chi_a ( -1 ) \chi_b ( p^m ) p^{ - m ( s + \alpha ) } e^{ - i \alpha n \pi } }{ m^2 n ( i m \log p - n \pi ) \log p }, \\
&E_5 ( s , \{ \chi_j \}_{ j = 1 }^2 ) \\
& \coloneqq \frac{ i }{ 2 } \sum_{ ( a , b ) \in \{ ( 1 , 2 ) , ( 2 , 1 ) \} } \sum_{ p , m } \frac{ \chi_a ( p^m ) p^{ - m \left( s - \frac{ 1 - \chi_b ( -1 ) }{ 2 } \right) } }{ m \sin ( i m \log p ) }, \\
&E_6 ( s , \{ \chi_j \}_{ j = 1 }^2 ) \\
& \coloneqq - \frac{ i }{ 2 \pi } \sum_{ ( a , b ) \in \{ ( 1 , 2 ) , ( 2 , 1 ) \} } \int_{ S^{ ( 2 ) } ( \pi \to 0 ) } \sum_{ p , m } p^{ - m ( s - u ) } ( \log p ) \chi_a ( p^m ) \log L ( u , \chi_b ) du, \\
&E_7 ( s , \{ \chi_j \}_{ j = 1 }^2 ) \\
& \coloneqq \frac{ 1 }{ 2 \pi } \sum_{ ( a , b ) \in \{ ( 1 , 2 ) , ( 2 , 1 ) \} } \left( \log \left( \frac{ \chi_a ( -1 ) \Gamma ( 1 + \alpha ) N_a^\alpha G ( \chi_a ) }{ ( 2 \pi )^{ 1 + \alpha } } \right) + \gamma + \log \left( \frac{ 2 \pi }{ N_a } + \frac{ \pi i }{ 2 } \right) \right) \\
& \phantom{ \frac{ 1 }{ 2 \pi } \sum_{ ( a , b ) \in \{ ( 1 , 2 ) , ( 2 , 1 ) \} } } \times \sum_{ p , m } \frac{ \chi_b ( p^m ) p^{ - m ( s + \alpha ) } }{ m^2 \log p } \\
& \quad + \sum_{ a = 1 }^2 \left( - \frac{ 1 + \alpha }{ 4 } \sum_{ p , m } \frac{ \chi_a ( p^m ) p^{ - m ( s + \alpha ) } }{ m } \right. \\
& \quad \left. \phantom{ + \sum_{ a = 1 }^2 } + \frac{ i }{ 2 \pi } \sum_{ p , m } \frac{ \chi_a ( p^m ) p^{ - m ( s + \alpha ) } }{ m^2 \log p } \int_0^\infty \frac{ 1 }{ e^u - 1 } \cdot \frac{ u + m ( \log p ) ( 1 - e^{ - \alpha u } ) }{ u + m \log p } du \right), \\
&E_8 ( s , \{ \chi_j \}_{ j = 1 }^2 ) \coloneqq \sum_{ ( a , b ) \in \{ ( 1 , 2 ) , ( 2 , 1 ) \} } \mutauaz \sum_{ p , m } \frac{ \chi_b ( p^m ) }{ m } p^{ - m \left( s - \frac{ 1 }{ 2 } - \taucaz \right) }, \\
&E_9 ( s , \{ \chi_j \}_{ j = 1 }^2 ) \coloneqq \sum_{ ( a , b ) \in \{ ( 1 , 2 ) , ( 2 , 1 ) \} } \mutauaz \sum_{ p , m } \frac{ \chi_b ( p^m ) }{ m } p^{ - m \left( s - \frac{ 1 }{ 2 } + \taucaz \right) }, \\
&E_{ 10 } ( s , \{ \chi_j \}_{ j = 1 }^2 ) \coloneqq \sum_{ ( a , b ) \in \{ ( 1 , 2 ) , ( 2 , 1 ) \} } \muaz \sum_{ p , m } \frac{ \chi_b ( p^m ) }{ m } p^{ - m \left( s - \frac{ 1 }{ 2 } \right) }.
\end{align*}
\end{thm}
The proofs of Theorem \ref{Thm1.1} and Theorem \ref{Thm1.3} are given in Section 5 and Section 6 respectively. The contents of the other sections are as follows. In Section 2 some lemmas are proved which are made use of in Section 3 or later. In Section 3 a series is introduced which includes information on the zeros of the Dirichlet $ L $-functions and some properties of the series is shown. In Section 4 the ``key equation'' is deduced.
\section*{\textbf{Acknowledgements}}
\noindent
I really thank Shin-ya Koyama for his giving me the opportunity to study this theme and much useful advice and Ki-ichiro Hashimoto for his special support. I also thank Hirotaka Akatsuka for his showing me the beneficial information for this study.
\section{\textbf{Lemmas}} \label{Sec2}
In this section, we prove some lemmas which are used later. 
\begin{lem} \label{Lem2.1}
Let $ c \in \mathbb{ C } - \{ 0 \} $ and $ \delta \in \mathbb{ R }_{ > 0 } $ be any fixed numbers. \\
{\upshape (i)} Suppose that $ f( u ) $ satisfies $ f( u ) = O( 1 ) \, ( u \to 0 ) $,  $ O( u^{ - \delta } ) \, ( u \to \infty ) $ and is holomorphic on $ \mathbb{ C } - \{ 0 \} $. Define
\begin{align*}
F_1( z ) \coloneqq \int_0^\infty \frac{ f( u ) }{ u - c z } du \quad ( \Im ( c z ) < 0 ).
\end{align*}
Then $ F_1 ( z ) + f( c z ) \log z $ is a single-valued meromorphic function of $ z $ on $ \mathbb{ C } - \{ 0 \} $. \\
{\upshape (ii)} Suppose that $ f( u ) $ satisfies $ f( u ) = O( 1 ) \, ( u \to 0 ) $, $ O( u^{ 1 - \delta } ) \, ( u \to \infty ) $ and is holomorphic on $ \mathbb{ C } - \{ 0 \} $. Define
\begin{align*}
F_2 ( z ) \coloneqq \int_0^\infty \frac{ f( u ) }{ u^2 - ( c z )^2 } du \quad ( \Im ( c z ) < 0 ).
\end{align*}
Then $ F_2 ( z ) + \frac{ f( c z ) - f( - c z ) }{ 2 c z } \log z $ is a single-valued meromorphic function of $ z $ on $ \mathbb{ C } - \{ 0 \} $.
\end{lem}
\begin{proof}[\underline{\textbf{Proof of Lemma \ref{Lem2.1}}}]
(i) If $ cz $ is in the fourth quadrant, then by Cauchy's theorem we have
\begin{align}
\lim_{ X \to \infty } \int_{ P_1 \cup P_2 } \frac{ f ( u ) }{ u - cz } du = 0, \label{PofLem2.1-1}
\end{align}
where
\begin{align*}
P_1 & \coloneqq \{ u \in \mathbb{ R } \, | \, 0 \le u \le X \}, \\
P_2 & \coloneqq \left\{ X e^{ i \varphi } \, \left| \, 0 \le \varphi \le \frac{ 3 \pi }{ 2 } \right. \right\} \cup \{ u \in \mathbb{ C } \, | \, \Re ( u ) = 0 , \, - X \le \Im ( u ) \le 0 \}
\end{align*}
for $ X \in \mathbb{ R }_{ >0 } $ and we go around the integral path in the counterclockwise direction. It follows from $ ( \ref{PofLem2.1-1} ) $ that
\begin{align}
F_1 ( z ) = - \lim_{ X \to \infty } \int_{ P_2 } \frac{ f ( u ) }{ u - cz } du. \label{PofLem2.1-2}
\end{align}
Since the integral path in the right-hand side of $ ( \ref{PofLem2.1-2} ) $ doesn't include the positive real axis, $ ( \ref{PofLem2.1-2} ) $ remains holomorphic while $ cz $ moving from the fourth quadrant into the first one across that axis. Therefore, $ ( \ref{PofLem2.1-2} ) $ gives the analytic continuation of $ F_1 ( z ) $ with $ cz $ in the first quadrant. On the other hand, when $ cz $ is in the first quadrant, by Cauchy's theorem we have
\begin{align*}
\lim_{ X \to \infty } \int_{ P_1 \cup P_2 } \frac{ f ( u ) }{ u - c z } du = 2 \pi i f ( c z ).
\end{align*}
From this and $ ( \ref{PofLem2.1-2} ) $ it follows that
\begin{align}
F_1 ( z ) = - 2 \pi i f( cz ) + \int_0^\infty \frac{ f ( u ) }{ u - cz } du. \label{PofLem2.1-3}
\end{align}
The right-hand side of $ ( \ref{PofLem2.1-3} ) $ is a holomorphic function of $ z $ if $ cz $ isn't on the non-negative real axis, so we find that $ F_1 ( z ) $ changes by $ - 2 \pi i f ( cz ) $ when $ cz $ moves counterclockwise around the origin, making one complete circuit. Therefore $ F_1 ( z ) + f ( cz ) \log z $ is unchanged by the analytic continuation around the origin，so a single-valued on $ \mathbb{ C } - \{ 0 \} $. \\
(ii)　When $ cz $ moves counterclockwise around the origin, the poles $ u = cz $ and $ u = - cz $ of the integrand of $ F_2 ( z ) $ cross the positive real axis by $ cz $ being moved from the fourth quadrant to the first one and by from the second one to the third one respectively. Therefore, we can show (ii) in a similar way as (i).
\end{proof}
\begin{rem} \label{Rem2.2}
If $ \Im ( cz ) > 0 $ then we have
\begin{align*}
F_2 ( z ) = - \pi i \frac{ h ( c z ) }{ c z } + \int_0^\infty \frac{ h ( u ) }{ u^2 - ( c z )^2 } du.
\end{align*}
We use this in the proof of Lemma \ref{Lem2.6}.
\end{rem}
Define that
\begin{alignat}{3}
& H ( t ) && \coloneqq \frac{ 1 }{ t } \int_0^\infty \frac{ 1 }{ e^u - 1 } \cdot \frac{ u - i t ( 1 - e^{ - \alpha u } ) }{ u - i t } du & \quad & ( \Re ( t ) < 0 ), \label{HDef} \\
& I_j ( t ) && \coloneqq \frac{ 1 }{ t } \int_0^\infty \frac{ u e^{ \frac{ 1 + \chi_j ( -1 ) }{ 4 } u } \left( \frac{ u }{ 2 } \cos \frac{ t }{ 2 } - t \sin \frac{ t }{ 2 } \right) }{ ( e^u - 1 )( u^2 + 4 t^2 ) } du && ( \Re ( t ) > 0 ), \label{IDef} \\
& J_j ( t ) && \coloneqq I_j ( t ) + \frac{ \log t }{ 4 \sin \frac{ t }{ 2 } } \label{JDef}
\end{alignat}
for $ j \in \mathbb{ Z }_{ >0 } $. For these functions, we show the following two lemmas : Lemma 2.4 and Lemma 2.6.
\begin{rem} \label{Rem2.3}
In the following, it is found that $ H ( t ) $, $ I_j ( t ) $ and $ J_j ( t ) $ has the analytic continuations, and let the same symbols denote those continuations respectively.
\end{rem}
\begin{lem} \label{Lem2.4}
{\upshape (i)} $ H ( t ) $ has the following asymptotic behavior at $ t = 0 $ :
\begin{align*}
H ( t ) = - \frac{ e^{ - i ( \alpha + \frac{ 1 }{ 2 } ) t } }{ 2  \sin \frac{ t }{ 2 } } \log t + O ( 1 ).
\end{align*}
{\upshape (ii)} $ H ( t ) + \frac{ e^{ - i ( \alpha + \frac{ 1 }{ 2 } ) t } }{ 2  \sin \frac{ t }{ 2 } } \log t $ is a single-valued meromorphic function on $ t \in \mathbb{ C } $. \\
{\upshape (iii)} $ H ( t ) $ has the simple pole at $ t = 2 n \pi \ ( n \in \mathbb{ Z } - \{ 0 \} ) $ with residue
\begin{align*}
\omega_n e^{ - \omega_n \cdot 2 \alpha \pi i } \left( \arg t - \frac{ 1 - \omega_n }{ 2 } \right),
\end{align*}
where $ \omega_n : = \frac{ n }{ | n | } $.
\end{lem}
\begin{rem} \label{Rem2.5}
If $ t \in \mathbb{ C } - i \mathbb{ R }_{ \le 0 } $ and the argument lies in $ \left( - \frac{ \pi }{ 2 } , \frac{ 3 \pi }{ 2 } \right) $, it follows from Lemma \ref{Lem2.4} {\upshape (ii)} that $ H ( t ) $ is a meromorphic function because $ \frac{ e^{ - i ( \alpha + \frac{ 1 }{ 2 } ) t } }{ 2 \sin \frac{ t }{ 2 } } \log t $ is such one.
\end{rem}
\begin{proof}[\underline{\textbf{Proof of Lemma \ref{Lem2.4}}}]
(i) $ ( \ref{HDef} ) $ is equivalent to
\begin{align}
t H ( t ) = \left( \int_0^1 + \int_1^\infty \right) \frac{ 1 }{ e^u - 1 } \cdot \frac{ u - i t ( 1 - e^{ - \alpha u } ) }{ u - i t } du \quad ( \Re ( t ) < 0 ). \label{PofLem2.4-1}
\end{align}
The second integral is holomorphic on $ t \in \mathbb{ C } - i \mathbb{ R }_{ \le - 1 } $ and particularly at $ t = 0 $ becomes
\begin{align}
\int_1^\infty \frac{ 1 }{ e^u - 1 } du. \label{PofLem2.4-2}
\end{align}
Next, we concider the first integral of the right-hand side of $ ( \ref{PofLem2.4-1} ) $. For $ | u | < 2 \pi $, we have
\begin{align}
\frac{ u - i t ( 1 - e^{ - \alpha u } ) }{ e^u - 1 } = \sum_{ n = 0 }^\infty a_n ( t ) u^n \label{PofLem2.4-3}
\end{align}
and then by the binomial theorem the right-hand side of $ (\ref{PofLem2.4-3}) $ becomes
\begin{align*}
& \sum_{ n = 0 }^\infty a_n ( t ) ( u - i t + i t )^n \\
& = \sum_{ n = 0 }^\infty a_n ( t ) ( i t )^n + \sum_{ n = 1 }^\infty a_n ( t ) ( u - i t )^n  + \sum_{ n = 2 }^\infty a_n ( t ) \sum_{ k = 1 }^{ n - 1 } \binom n k ( u - i t )^{ n - k } ( i t )^k
\end{align*}
so
\begin{align}
& ( \mbox{the first integral of the right-hand side of } ( \ref{PofLem2.4-1} ) ) \notag \\
\begin{split}
& = \sum_{ n = 0 }^\infty a_n ( t ) ( i t )^n \int_0^1 \frac{ 1 }{ u - i t } du + \sum_{ n = 1 }^\infty a_n ( t ) \int_0^1 ( u - i t )^{ n - 1 } du \\
& \quad + \sum_{ n = 2 }^\infty a_n ( t ) \sum_{ k = 1 }^{ n - 1 } \binom n k ( i t )^k \int_0^1 ( u - i t )^{ n - k - 1 } du. 
\end{split} \label{PofLem2.4-4}
\end{align}
The third term of $ ( \ref{PofLem2.4-4} ) $ is holomorphic for $ | t | < 1 $ and vanishes at $ t = 0 $. The second term of $ ( \ref{PofLem2.4-4} ) $ is equal to
\begin{align}
\int_0^1 \sum_{ n = 1 }^\infty a_n ( 0 ) u^{ n - 1 } du = \int_0^1 \left( \frac{ 1 }{ e^u - 1 } - \frac{ 1 }{ u } \right) du \label{PofLem2.4-5}
\end{align}
at $ t = 0 $, where we use $ a_0 ( 0 ) = 1 $ because we have
\begin{align*}
\frac{ u }{ e^u - 1 } = \sum_{ n = 0 }^\infty a_n ( 0 ) u^n
\end{align*}
from $ ( \ref{PofLem2.4-3} ) $. Then, the first term of $ ( \ref{PofLem2.4-4} ) $ is equal to
\begin{align*}
\sum_{ n = 0 }^\infty a_n ( t ) ( i t )^n ( \log ( 1 - i t ) - \log ( - i t ) ) = - \frac{ i t e^{ - i \alpha t } }{ e^{ i t } - 1 } \log t + t h_1 ( t )
\end{align*}
for $ | t | < 1 $, where $ h_1 ( t ) $ is a power series which converges for $ | t | < 1 $. Noting that by Cram\'{e}r~\cite[p. 117, (20)]{Cra} it was shown that
\begin{align*}
( \mbox{the right-hand side of } ( \ref{PofLem2.4-5} ) ) + ( \ref{PofLem2.4-2} ) = 0,
\end{align*}
it follows that
\begin{align}
H ( t ) = - \frac{ i e^{ - i \alpha t } }{ e^{ i t } - 1 } \log t + h_2 ( t ) = - \frac{ e^{ - i ( \alpha + \frac{ 1 }{ 2 } ) t } }{ 2 \sin \frac{ t }{ 2 } } \log t + h_2 ( t ), \label{PofLem2.4-6}
\end{align}
where $ h_2 ( t ) $ is a power series which converges in $ | t | < 1 $. The proof of (i) is complete. \\
(ii) We find that $ H ( t ) + \frac{ e^{ - i ( \alpha + \frac{ 1 }{ 2 } ) t } }{ 2 \sin \frac{ t }{ 2 } } \log t $ is holomorphic for $ | t | < 1 $ from $ ( \ref{PofLem2.4-6} ) $ and is a single-valued function on $ t \in \mathbb{ C } - \{ 0 \} $ from Lemma \ref{Lem2.1} (i). The proof of (ii) is complete. \\
(iii) By $ ( \ref{HDef} ) $ it is easily found that $ H ( t ) $ is holomorphic if $ \arg t \in \left( - \frac{ \pi }{ 2 } , \frac{ 3 \pi }{ 2 } \right) $. From this and Lemma \ref{Lem2.4} (ii) we can obtain the desired result.
\end{proof}
\begin{lem} \label{Lem2.6}
$ J_j ( t ) $ has the following properties : \\
{\upshape (i)} $ J_j ( t ) $ is a single-valued meromorphic function on $ t \in \mathbb{ C } - \{ 0 \} $, \\
{\upshape (ii)} $ J_j ( t ) $ satisfies that for $ t \in \mathbb{ C } - i \mathbb{ R }_{ \le 0 } $
\begin{align}
J_j( t ) + J_j ( - t ) = \begin{cases}
\dsp{ - \pi i \frac{ e^{ - \frac{ \chi_j ( -1 ) }{ 2 } it } }{ 2 \sin t } - \frac{ i \pi }{ 4 \sin \frac{ t }{ 2 } } } & ( \Re ( t ) < 0 ), \\[+10pt]
\dsp{ \pi i \frac{ e^{ \frac{ \chi_j ( -1 ) }{ 2 } it } }{ 2 \sin t } + \frac{ i \pi }{ 4 \sin \frac{ t }{ 2 } } } & ( \Re ( t ) > 0 ),
\end{cases} \label{Lem2.6-1}
\end{align}
where the argument lies in $ \left( - \frac{ \pi }{ 2 } , \frac{ 3 \pi }{ 2 } \right) $.
\end{lem}
\begin{proof}[\underline{\textbf{Proof of Lemma \ref{Lem2.6}}}]
(i) We should use Lemma \ref{Lem2.1} (ii) as
\begin{align*}
c = - 2i , \: z = t \mbox{ and } f( u ) = \frac{ u e^{ \frac{ 1 + \chi_j ( - 1 ) }{ 4 } u } \left( \frac{ u }{ 2 } \cos \frac{ t }{ 2 } - t \sin \frac{ t }{ 2 } \right) }{ t ( e^u - 1 ) }.
\end{align*}
(ii) Let $ t \in \mathbb{ C } - i \mathbb{ R }_{ \le 0 } $ and the argument lie in $ \left( - \frac{ \pi }{ 2 }, \frac{ 3 \pi }{ 2 } \right) $. Then, since $ \frac{ \log t }{ 4 \sin \frac{ t }{ 2 } } $ is meromorphic, by Lemma \ref{Lem2.6} (i) we find that $ I_j ( t ) $ is as well.

Now, by Remark \ref{Rem2.2}, for $ \Re ( t ) < 0 $ we have
\begin{align}
I_j ( t ) = - \pi i \frac{ e^{ \frac{ 1 + \chi_j ( - 1 ) }{ 2 } i t } }{ \sin t } + \frac{ 1 }{ t } \int_0^\infty \frac{ u e^{ \frac{ 1 + \chi_j ( -1 ) }{ 4 } u } \left( \frac{ u }{ 2 } \cos \frac{ t }{ 2 } - t \sin \frac{ t }{ 2 } \right) }{ ( e^u - 1 ) ( u^2 + 4 t^2 ) } du. \label{PofLem2.6-1}
\end{align}
Adding $ \frac{ \log t }{ 4 \sin \frac{ t }{ 2 } } $ to the both sides of $ ( \ref{PofLem2.6-1} ) $, we obtain
\begin{align}
J_j ( t ) = ( \mbox{the right-hand side of } ( \ref{PofLem2.6-1} ) ) + \frac{ \log t }{ 4 \sin \frac{ t }{ 2 } } \quad ( \Re ( t ) < 0 ). \label{PofLem2.6-2}
\end{align}
On the other hand, from $ ( \ref{IDef} ) $ and $ ( \ref{JDef} ) $ we can obtain the following equations : for $ \Re ( t ) < 0 $
\begin{align}
\begin{split}
J_j ( - t ) & = I_j ( - t ) - \frac{ \log ( - t ) }{ 4 \sin \frac{ t }{ 2 } } \\
& = - \frac{ 1 }{ t } \int_0^\infty \frac{ u e^{ \frac{ 1 + \chi_j ( - 1 ) }{ 4 } u } \left( \frac{ u }{ 2 } \cos \frac{ t }{ 2 } - t \sin \frac{ t }{ 2 } \right) }{ ( e^u - 1 ) ( u^2 + 4 t^2 ) } du - \frac{ \log t }{ 4 \sin \frac{ t }{ 2 } } - \frac{ i \pi }{ 4 \sin \frac{ t }{ 2 } }.
\end{split} \label{PofLem2.6-3}
\end{align}
From $ ( \ref{PofLem2.6-2} ) $ and $ ( \ref{PofLem2.6-3} ) $ it is follows that
\begin{align}
J_j ( t ) + J_j ( - t ) = - \pi i \frac{ e^{ 1 + \chi_j ( -1 ) }{ 4 } i t }{ \sin t } - \frac{ i \pi }{ 4 \sin \frac{ t }{ 2 } } \quad ( \Re ( t ) < 0 ), \label{PofLem2.6-4}
\end{align}
which is $ ( \ref{Lem2.6-1} ) $ with $ \mathbb{Re} ( t ) < 0 $. By replacing $ t $ with $ - t $ in $ ( \ref{PofLem2.6-4} ) $, we obtain $ ( \ref{Lem2.6-1} ) $ with $ \Re ( t ) > 0 $.
\end{proof}
Lemma \ref{Lem2.7} was proved by Akatsuka~\cite{Aka2}.
\begin{lem} \label{Lem2.7}
{\upshape (i)~\cite[Lemma 2.5]{Aka2}~}For any $ X , Y \in \mathbb{ R }_{ > 0 } $ satisfying $ X < Y $
\begin{align*}
\log Y - \log X \ge \frac{ Y - X }{ Y }.
\end{align*}
{\upshape (ii)~\cite[Remark 2.1]{Aka2}~}$ \dsp{ \sum_{ p } \sum_{ m = 1 }^\infty \frac{ p^{ - m } }{ m^2 \log p } } < \infty $. \\
{\upshape (iii)~\cite[p639, (4.4)]{Aka2}~}For any fixed $ \delta \in \mathbb{ R }_{ > 0 } $ and any $ A \in \mathbb{ R } $
\begin{align*}
\sum_{ p } \sum_{ m = 1 }^\infty p^{ - m ( 1 + \delta ) } ( m \log p )^A \log p < \infty.
\end{align*}
\end{lem}
We prove a formula for the gamma function in the following lemma.
\begin{lem} \label{Lem2.8}
Let any fixed $ \psi \in \mathbb{ R } $ satisfy $ - \frac{ \pi }{ 2 } < \psi < \frac{ \pi }{ 2 } $ and let $ \arg \nu \in ( - \psi - \frac{ \pi }{ 2 } , - \psi + \frac{ \pi }{ 2 } ) $ and $ \Re ( w ) > 0 $. Then, we have
\begin{align*}
\frac{ \Gamma ( w ) }{ \nu^w } = \int_0^{ \infty e^{ i \psi } } e^{ - \nu t } t^w \frac{ dt }{ t }.
\end{align*}
\end{lem}
\begin{proof}[\underline{\textbf{Proof of Lemma \ref{Lem2.8}}}]
For any fixed $ \psi \in \mathbb{ R } $ satisfying $ - \frac{ \pi }{ 2 } < \psi < \frac{ \pi }{ 2 } $, let $ \arg \nu = - \psi $. Then, we have
\begin{align*}
\frac{ \Gamma ( w ) }{ \nu^w } = \int_0^\infty e^{ - t } \left( \frac{ t }{ \nu } \right)^w \frac{ d t }{ t } = \int_0^{ \infty \nu^{ -1 } } e^{ - \nu t } t^w \frac{ d t }{ t } = \int_0^{ \infty e^{ i \psi } } e^{ - \nu t } t^w \frac{ d t }{ t }.
\end{align*}
When $ w $ is fixed in $ \Re ( w ) > 0 $, the both sides are holomorphic in
\begin{align*}
\{ \nu \in \mathbb{ C } \ | \ \Re ( \nu e^{ i \psi } ) > 0 \} = \left\{ \nu \in \mathbb{ C } \ \left| \ -  \psi - \frac{ \pi }{ 2 } < \arg \nu < \frac{ \pi }{ 2 } - \psi \right. \right\}.
\end{align*}
This completes the proof.
\end{proof}
\section{\textbf{Properties of a series concerning the zeros of the Dirichlet $ \bm{ L } $-functions}} \label{Sec3}
For a series $ \theta ( t ) \coloneqq \sum_{ \Re ( \tau ) } e^{ - \tau t } \, ( \Re ( t ) > 0 ) $ where $ \tau \in \mathbb{ C } $ with $ \rho = \frac{ 1 }{ 2 } + i \tau $ for the imaginary zeros $ \rho $ of the Riemann zeta function, Cram\'{e}r~\cite{Cra} deduced the explicit formula and then Guinand~\cite{Gui} obtained the meromorphic continuation and the poles by proving the functional equation and deduced the approximate behavior. Akatsuka~\cite{Aka2} introduced $ \theta^{ * } ( t ) \coloneqq \theta ( t ) - e^{ - \frac{ i t }{ 2 } } \, ( t \in \mathbb{ C } - i \mathbb{ R }_{ \le 0 }  ) $ and proved the properties on the basis of the results of Cram\'{e}r and Guinand.

We define a following series :
\begin{align}
l_{ \chi_j } ( t ) \coloneqq \sum_{ \Re ( \taucj ) > 0 } e^{ - \taucj t } \quad ( \Re ( t ) > 0 ). \label{lcDef}
\end{align}
for $ j \in \mathbb{ Z }_{ >0 } $. With reference to the methods of the above three mathematicians we research in this series.

We define the complete Dirichlet $ L $-function $ \hat{ L } ( s , \chi_j ) $ by
\begin{align*}
\hat{ L } ( s , \chi_j ) \coloneqq \left( \frac{ \pi }{ N_j } \right)^{ - \left( \frac{ s }{ 2 } + \frac{ 1 - \chi_j ( - 1 ) }{ 4 } \right) } \Gamma \left( \frac{ s }{ 2 } + \frac{ 1 - \chi_j ( - 1 ) }{ 4 } \right) L ( s , \chi_j )
\end{align*}
and define that
\begin{align*}
\xi ( s , \chi_j ) \coloneqq \hat{ L } \left( s + \frac{ 1 }{ 2 } , \chi_j \right).
\end{align*}
It is well known that $ \dsp{ \frac{ \xi^\prime }{ \xi } ( s , \chi_j ) } $ satisfies the functional equation :
\begin{align*}
\frac{ \xi^\prime }{ \xi } ( - s , \chi_j ) = - \frac{ \xi^\prime }{ \xi } ( s , \bar{ \chi_j } ).
\end{align*}

First, we prove the meromorphy and the functional equation of $ l_{ \chi_j } ( t ) $.
\begin{thm} \label{Thm3.1}
$ l_{ \chi_j } ( t ) $ has a meromorphic continuation to $ \mathbb{ C } - i \mathbb{ R }_{ \le 0 } $ for which
\begin{align}
& l_{ \chi_j } ( t ) + l_{ \bar{ \chi_j } } ( - t ) \notag \\
& = \begin{cases}
\dsp{ - \frac{ i e^{ - \frac{ \chi_j ( - 1 ) }{ 2 } it } }{ 2 \sin t } - \mutaujz ( e^{ i \taucjz t } + e^{ - i \taucjz t } ) - \mujz } & ( \Re ( t ) < 0 ), \\[+15pt]
\dsp{ \frac{ i e^{ \frac{ \chi_j ( -1 ) }{ 2 } it } }{ 2 \sin t } - \mutaujz ( e^{ i \taucjz t } + e^{ - i \taucjz t } ) - \mujz } & ( \Re ( t ) > 0 ),
\end{cases} \label{Thm3.1-1}
\end{align}
where the argument lies in $ ( - \frac{ \pi }{ 2 } , \ \frac{ 3 \pi }{ 2 } ) $.
\end{thm}
\begin{proof}[\underline{\textbf{Proof of Theorem \ref{Thm3.1}}}]
If $ \Re ( t ) > 0 $, then by Cauchy's theorem we have
\begin{align}
l_{ \chi_j } ( t ) &= \frac{ 1 }{ 2 \pi i } \int_{ C_1 \cup C_{ 2 , j } \cup C_3 } e^{ i s t } \frac{ \xi^\prime }{ \xi } ( s , \chi_j ) ds \notag \\
&= \frac{ 1 }{ 2 \pi i } \left( \int_{ C_1 } + \int_{ C_{ 2 , j } } + \int_{ C_3 } \right) e^{ i s t } \frac{ \xi^\prime }{ \xi } ( s , \chi_j ) ds, \label{PofThm3.1-1}
\end{align}
where
\begin{alignat*}{2}
&C_1 && \coloneqq \left\{ s \in \mathbb{ C } \ \left| \ \Re ( s ) = - \frac{ 1 }{ 2 } , \, \Im ( s ) \ge 0 \right. \right\} \\
&C_{ 2 , j } && \coloneqq \left\{ \left. \frac{ 1 }{ 2 } \cos \varphi + i \varepsilon_j \sin \varphi \ \right| 0 \le \varphi \le \pi  \right\} \\
&C_3 && \coloneqq \left\{ s \in \mathbb{ C } \ \left| \Re ( s ) = \frac{ 1 }{ 2 } , \, \Im ( s ) \ge 0 \right. \right\},
\end{alignat*}
and we go around the integral path in the counterclockwise direction.

First, we consider the integral of the path $ C_1 $. It becomes
\begin{align}
\int_{ C_1 } &= \int_\infty^0 e^{ i \left( - \frac{ 1 }{ 2 } + i y \right) t } \frac{ \xi^\prime }{ \xi } \left( - \frac{ 1 }{ 2 } + i y , \chi_j \right) i dy \notag \\
&= - i e^{ - \frac{ i t }{ 2 } } \int_\infty^0 e^{ - y t } \frac{ \xi^\prime }{ \xi } \left( \frac{ 1 }{ 2 } - i y , \bar{ \chi_j } \right) dy \notag \\
&= i e^{ - \frac{ i t }{ 2 } } \int_0^\infty e^{ - yt } \left( - \frac{ 1 }{ 2 } \log \left( \frac{ \pi }{ N_j } \right) \right.  \label{PofThm3.1-2} \\
& \phantom{ i e^{ - \frac{ i t }{ 2 } } \int_0^\infty e^{ - yt } \quad } + \frac{ 1 }{ 2 } \frac{ \Gamma^\prime }{ \Gamma } \left( \frac{ 3 - \chi_j ( -1 ) }{ 4 } - \frac{ i y }{ 2 } \right) \label{PofThm3.1-3} \\
& \phantom{ i e^{ - \frac{ i t }{ 2 } } \int_0^\infty e^{ - yt } \quad } \left. + \frac{ L^\prime }{ L } ( 1 - iy , \bar{ \chi_j } ) \right) dy. \label{PofThm3.1-4}
\end{align}
The term concerning $ ( \ref{PofThm3.1-2} ) $ becomes
\begin{align*}
- \frac{ i }{ 2 } e^{ - \frac{ i t }{ 2 } } \log \left( \frac{ \pi }{ N_j } \right) \int_0^\infty e^{ - y t } dy = - \frac{ i }{ 2 t } e^{ - \frac{ i t }{ 2 } } \log \left( \frac{ \pi }{ N_j } \right).
\end{align*}
Concerning $ ( \ref{PofThm3.1-3} ) $, since
\begin{align*}
\frac{ \Gamma^\prime }{ \Gamma } ( s ) = \int_0^\infty \left( \frac{ e^{ - u } }{ u } - \frac{ e^{ - ( s - 1 ) u } }{ e^u - 1 } \right) du \quad ( \Re ( s ) > 0 ),
\end{align*}
it follows that
\begin{align}
& \frac{ i }{ 2 } \int_0^\infty e^{ - y t } \int_0^\infty \left( \frac{ e^{ - u } }{ u } - \frac{ e^{ \left( \frac{ 1 + \chi_j ( - 1 ) }{ 4 } + \frac{ i y }{ 2 } \right) u } }{ e^u - 1 } \right) du dy \notag \\
& = \frac{ i }{ 2 } e^{ - \frac{ i t }{ 2 } } \int_0^\infty \left( \frac{ e^{ - u } }{ u } \int_0^\infty e^{ - y t } dy - \frac{ e^{ \frac{ 1 + \chi_j ( - 1 ) }{ 4 } u } }{ e^u - 1 } \int_0^\infty e^{ \left( - t + \frac{ i u }{ 2 } \right) y } dy \right) du \notag \\
& = \frac{ i }{ 2 } e^{ - \frac{ i t }{ 2 } } \int_0^\infty \left( \frac{ e^{ - u } }{ u t } - \frac{ e^{ \frac{ 1 + \chi_j ( - 1 ) }{ 4 } u } }{ ( e^u - 1 ) \left( t - \frac{ i u }{ 2 } \right) } \right) du. \label{PofThm3.1-5}
\end{align}
Concerning $ ( \ref{PofThm3.1-4} ) $, since the Euler product $ \prod_{p} ( 1 - \chi_j ( p )p^{ - s } )^{ -1 } ( = L ( s , \chi_j ) ) $ converges uniformly on $ \Re ( s ) = 1 $, we have
\begin{align*}
& i e^{ - \frac{ it }{ 2 } } \int_0^\infty e^{ - y t } \left( - \sum_{ p } \sum_{ m = 1 }^\infty \bar{ \chi_j } ( p^m ) ( \log p ) p^{ - m ( 1 - i y ) } \right) dy \\
& = - i e^{ - \frac{ i t }{ 2 } } \sum_{ p } \sum_{ m = 1 }^\infty \bar{ \chi_j } ( p^m ) p^{ - m } ( \log p ) \int_0^\infty e^{ ( - t + i m \log p ) y } dy \\
& = - i e^{ - \frac{ i t }{ 2 } } \sum_{ p } \sum_{ m = 1 }^\infty \frac{ \bar{ \chi_j } ( p^m ) p^{ - m } ( \log p ) }{ t - i m \log p }.
\end{align*}

Similarly, we can calculate the integral of the path $ C_3 $ in $ ( \ref{PofThm3.1-1} ) $ and obtain the following result :
\begin{align}
\int_{ C_3 } & = - \frac{ i }{ 2 t } e^{ \frac{ i t }{ 2 } } \log \left( \frac{ \pi }{ N } \right) \notag \\
& \quad + \frac{ i }{ 2 } e^{ \frac{ i t }{ 2 } } \displaystyle{ \int_0^\infty \left( \frac{ e^{ - u } }{ u t } - \frac{ e^{ \frac{ 1 + \chi_j ( - 1 ) }{ 4 } u } }{ ( e^u - 1 ) \left( t + \frac{ i u }{ 2 } \right) } \right) du } \label{PofThm3.1-6} \\
& \quad - i e^{ \frac{ i t }{ 2 } } \displaystyle{ \sum_{ p } \sum_{ m = 1 }^\infty \frac{ \chi_j ( p^m ) p^{ - m } \log p }{ t + i m \log p } }. \notag
\end{align}
Noting that
\begin{align*}
& ( \ref{PofThm3.1-5} ) + ( \ref{PofThm3.1-6} ) \\
& = \frac{ i }{ 2 } \left( \frac{ 2 \cos \frac{ t }{ 2 } }{ t } \int_0^\infty \frac{ e^{ - u } }{ u } du - \int_0^\infty \frac{ e^{ \frac{ 1 + \chi_j ( - 1 ) }{ 4 } u } }{ e^u - 1 } \left( \frac{ e^{ - \frac{ i t }{ 2 } } }{ t - \frac{ i u }{ 2 } } + \frac{ e^{ \frac{ i t }{ 2 } } }{ t + \frac{ i u }{ 2 } } \right) du \right) \\
& = \frac{ i }{ 2 } \left( \frac{ 2 \cos \frac{ t }{ 2 } }{ t } \left( \int_0^\infty \frac{ e^{ - u } }{ u } du -  \int_0^\infty \frac{ e^{ \frac{ 1 + \chi_j ( - 1 ) }{ 4 } u } }{ e^u - 1 } du + \int_0^\infty \frac{ e^{ \frac{ 1 + \chi_j ( - 1 ) }{ 4 } } }{ e^u - 1 } du \right) \right. \\
& \left. \qquad - \int_0^\infty \frac{ e^{ \frac{ 1 + \chi_j ( - 1 ) }{ 4 } u } }{ e^u - 1 } \cdot \frac{ 2 t \cos \frac{ t }{ 2 } + u \sin \frac{ t }{ 2 } }{ t^2 + \frac{ u^2 }{ 4 } } du \right) \\
& = \frac{ i \cos \frac{ t }{ 2 } }{ t } \int_0^\infty \left( \frac{ e^{ - u } }{ u } - \frac{ e^{ \frac{ 1 + \chi_j ( -1 ) }{ 4 } u } }{ e^u - 1 } \right) du +2i I_j ( t ),
\end{align*}
we can deduce that for $ \Re ( t ) > 0 $
\begin{align*}
l_{ \chi_j } ( t ) & = - \frac{ \cos \frac{ t }{ 2 } }{ 2 \pi t } \log \left( \frac{ \pi }{ N_j } \right) + \frac{ \cos \frac{ t }{ 2 } }{ 2 \pi t } \int_0^\infty \left( \frac{ e^{ - u } }{ u } - \frac{ e^{ \frac{ 1 + \chi_j ( -1 ) }{ 4 } u } }{ e^u - 1 } \right) du + \frac{ 1 }{ \pi } I_j ( t ) \\
& \quad - \frac{ 1 }{ 2 \pi } e^{ - \frac{ i t }{ 2 } } \sum_{ p } \sum_{ m = 1 }^\infty \frac{ \bar{ \chi_j } ( p^m ) p^{ - m } \log p }{ t - i m \log p } - \frac{ 1 }{ 2 \pi } e^{ \frac{ i t }{ 2 } } \sum_{ p } \sum_{ m = 1 }^\infty \frac{ \chi_j ( p^m ) p^{ - m } \log p }{ t + i m \log p } \\
& \quad + \frac{ 1 }{ 2 \pi i } \int_{ C_{ 2 , j } ( \pi \to 0 ) } e^{ i s t } \frac{ \xi^\prime }{ \xi } ( s , \chi_j ) ds.
\end{align*}
Adding $ \frac{ \log t }{ 4 \pi \sin \frac{ t }{ 2 } } $ to the both sides, we have
\begin{align}
\left.
\begin{array}{@{}l}
\dsp{ l_{ \chi_j } ( t ) + \frac{ \log t }{ 4 \pi \sin \frac{ t }{ 2 } } } \\
= - \dsp{ \frac{ \cos \frac{ t }{ 2 } }{ 2 \pi t } \log \left( \frac{ \pi }{ N_j } \right) + \frac{ \cos \frac{ t }{ 2 } }{ 2 \pi t } \int_0^\infty \left( \frac{ e^{ - u } }{ u } - \frac{ e^{ \frac{ 1 + \chi_j ( -1 ) }{ 4 } u } }{ e^u - 1 } \right) du + \frac{ 1 }{ \pi } J_j ( t ) } \\
\quad \dsp{ - \frac{ 1 }{ 2 \pi } e^{ - \frac{ i t }{ 2 } } \sum_{ p } \sum_{ m = 1 }^\infty \frac{ \bar{ \chi_j } ( p^m ) p^{ - m } \log p }{ t - i m \log p } - \frac{ 1 }{ 2 \pi } e^{ \frac{ i t }{ 2 } } \sum_{ p } \sum_{ m = 1 }^\infty \frac{ \chi_j ( p^m ) p^{ - m } \log p }{ t + i m \log p } } \\
\quad \dsp{ + \frac{ 1 }{ 2 \pi i } \int_{ C_{ 2 , j } ( \pi \to 0 ) } e^{ i s t } \frac{ \xi^\prime }{ \xi } ( s , \chi_j ) ds }
\end{array} \right\} \label{PofThm3.1-7}
\end{align}
because $ J_j ( t ) = I_j ( t ) + \frac{ \log t }{ 4 \sin \frac{ t }{ 2 } } $. Therefore, it is found by Lemma $ \ref{Lem2.6} $ that $ l_{ \chi_j } ( t ) + \frac{ \log t }{ 4 \pi \sin \frac{ t }{ 2 } } $ is a sigle-valued meromorphic function on $ t \in \mathbb{ C } - \{ 0 \} $ and further if $ \arg t \in ( - \frac{ \pi }{ 2 } , \frac{ 3 \pi }{ 2 } ) $ then it follows from the meromorphy of $ \frac{ \log t }{ 4 \pi \sin \frac{ t }{ 2 } } $ that $ l_{ \chi_j } ( t ) $ is also a meromorphic function. It implies  the former statement of Theorem \ref{Thm3.1}. In the following, the argument lies in $ ( - \frac{ \pi }{ 2 } , \frac{ 3 \pi }{ 2 } ) $.

Now, let $ \Re ( t ) < 0 $. Replacing $ t $ and $ \chi_j $ with $ - t $ and $ \bar{ \chi_j } $ respectively in $ ( \ref{PofThm3.1-7} ) $, we have
\begin{align}
\left. \begin{array}{@{}l}
\dsp{ l_{ \bar{ \chi_j } } ( - t ) - \frac{ \log ( - t ) }{ 4 \pi \sin \frac{ t }{ 2 } } } \\
= \dsp{ \frac{ \cos \frac{ t }{ 2 } }{ 2 \pi t } \log \left( \frac{ \pi }{ N_j } \right) - \frac{ \cos \frac{ t }{ 2 } }{ 2 \pi t } \int_0^\infty \left( \frac{ e^{ - u } }{ u } - \frac{ e^{ \frac{ 1 + \chi_j ( -1 ) }{ 4 } u } }{ e^u - 1 } \right) du + \frac{ 1 }{ \pi } J_j ( - t ) } \\
\quad \dsp{ + \frac{ 1 }{ 2 \pi } e^{ \frac{ i t }{ 2 } } \sum_{ p } \sum_{ m = 1 }^\infty \frac{ \chi_j ( p^m ) p^{ - m } \log p }{ t + i m \log p } + \frac{ 1 }{ 2 \pi } e^{ - \frac{ i t }{ 2 } } \sum_{ p } \sum_{ m = 1 }^\infty \frac{ \bar{ \chi_j } ( p^m ) p^{ - m } \log p }{ t - i m \log p } } \\
\quad \dsp{ + \frac{ 1 }{ 2 \pi i } \int_{ C_{ 2, j } ( \pi \to 0 ) } e^{ - i s t } \frac{ \xi^\prime }{ \xi } ( s , \bar{ \chi_j } ds ) }.
\end{array} \right\} \label{PofThm3.1-8}
\end{align}
The second term of the left-hand side of $ ( \ref{PofThm3.1-8} ) $ is equal to
\begin{align*}
- \frac{ \log t }{ 4 \pi \sin \frac{ t }{ 2 } } - \frac{ i }{ 4 \sin \frac{ t }{ 2 } }.
\end{align*}
By Lemma \ref{Lem2.6} (ii), the third term of the right-hand side of $ ( \ref{PofThm3.1-8} ) $ is equal to
\begin{align*}
-i \frac{ e^{ - \frac{ \chi_j ( - 1 ) }{ 2 } i t } }{ 2 \sin t } - \frac{ i }{ 4 \sin \frac{ t }{ 2 } } - \frac{ 1 }{ \pi } J_j ( t ).
\end{align*}
By replacing $ s $ with $ - s $, we find that the last term of the right-hand side of $ ( \ref{PofThm3.1-8} ) $ becomes
\begin{align*}
\frac{ 1 }{ 2 \pi i } \int_{ C_{ 4 , j } ( 0 \to - \pi ) } e^{ i s t } \frac{ \xi^\prime }{ \xi } ( - s , \bar{ \chi_j } ) ( - ds ) = \frac{ 1 }{ 2 \pi i } \int_{ C_{ 4 , j } ( 0 \to - \pi ) } e^{ i s t } \frac{ \xi^\prime }{ \xi } ( s , \chi_j ) ds,
\end{align*}
where
\begin{align*}
C_{ 4 , j } \coloneqq \left\{ \left. \frac{ 1 }{ 2 } cos \varphi + i \varepsilon_j \sin \varphi \ \right| \ \varphi \in \mathbb{ R } , \ - \pi \le \varphi \le 0 \right\}.
\end{align*}
Hence, adding $ ( \ref{PofThm3.1-7} ) $ and $ ( \ref{PofThm3.1-8} ) $, it follows that
\begin{align}
l_{ \chi_j } ( t ) + l_{ \bar{ \chi_j } } ( - t ) = -\frac{ i e^{ - \frac{ \chi_j ( - 1 ) }{ 2 } i t } }{ 2 \sin t } + \frac{ 1 }{ 2 \pi i } \int_{ C_2 ( \pi \to 0 ) \cup C_4 ( 0 \to - \pi ) } e^{ i s t } \frac{ \xi^\prime }{ \xi } ( s , \chi_j ) ds. \label{PofThm3.1-9}
\end{align}
By residue theorem, the second term of the right-hand side of $ ( \ref{PofThm3.1-9} ) $ becomes
\begin{align*}
- \mutaujz e^{ i \taucjz t } - \mutaujz e^{ - i \taucjz t } - \mujz.
\end{align*}
Hence, we obtain $ ( \ref{Thm3.1-1} ) $ for $ \Re ( t ) < 0 $. We can obtain $ ( \ref{Thm3.1-1} ) $ for $ \Re ( t ) > 0 $ by replacing $ t $ with $ - t $ in $ ( \ref{Thm3.1-1} ) $ for $ \Re ( t ) < 0 $. This completes the proof.
\end{proof}
Next, we deduce the explicit formula, the approximate behavior and the poles of $ l_{ \chi_j } ( t ) $.
\begin{thm} \label{Thm3.2}
Define $ S_j : = \left\{ \frac{ 1 + \alpha }{ 2 } \cos \varphi + i \varepsilon_j \sin \varphi + \frac{ 1 - \alpha }{ 2 } \mid 0 \le \varphi \le \pi \right\} $. \\
{\upshape (i)} $ l_{ \chi_j } ( t ) $ has the following expression for $ \Re ( t ) > 0 $ :
\begin{align*}
& l_{ \chi_j } ( t ) \\
&= - \frac{ it }{ 2 \pi } e^{ \frac{ it }{ 2 } } \sum_p \sum_{ m = 1 }^\infty \frac{ \chi_j( p^m ) p^{ - m } }{ m ( t + im \log p ) } + \frac{ e^{ - i \left( \alpha + \frac{ 1 }{ 2 } \right) t } }{ 2 \pi } \left( it \sum_p \sum_{ m = 1 }^\infty \frac{ \bar{ \chi_j } ( p^m ) p^{ - m ( 1 + \alpha ) } }{ m ( t - im \log p ) } \right. \\
&\quad - it \sum_{ m = 1 }^\infty \frac{ \chi_j( -1 )^m e^{ - \alpha m \pi i } }{ m ( t + m \pi ) } +  i \log \left( \frac{ \chi_j( -1 ) \Gamma( 1 + \alpha ) N_j^\alpha G( \chi_j ) }{ ( 2 \pi )^{ 1 + \alpha } } \right) - \frac{ ( 1 + \alpha ) \pi }{ 2 }  \\
& \quad \left. - \frac{ 1 }{ t } \left( \gamma + \log \left( \frac{ 2 \pi }{ N_j } \right) + \frac{ \pi i }{ 2 } \right) + \frac{ 1 }{ t } \int_0^\infty \frac{ 1 }{ e^u - 1 } \cdot \frac{ u + it ( 1 - e^{ - \alpha u } ) }{ u + it } du \right) \\
& \quad - \frac{ t }{ 2 \pi } e^{ - \frac{ it }{ 2 } } \int_{ S_j ( \pi \to 0 ) } e^{ ist } \log L ( s , \chi_j ) ds.
\end{align*}
{\upshape (ii)} $ l_{ \chi_j } ( t ) $ has the following expression for $ t \in \mathbb{ C } - i \mathbb{ R }_{ \le 0 } $ :
\begin{align}
\left.
\begin{array}{@{}l}
l_{ \chi_j } ( t ) \\
= \dsp{ - \frac{ it }{ 2 \pi } e^{ - \frac{ it }{ 2 } } \sum_{ p } \sum_{ m = 1 }^\infty \frac{ \bar{ \chi_j } ( p^m ) p^{ - m } }{ m ( t - im \log p ) } - \frac{ e^{ i \left( \alpha + \frac{ 1 }{ 2 } \right) t } }{ 2 \pi } \left( it \sum_{ p } \sum_{ m = 1 }^\infty \frac{ \chi_j ( p^m ) p^{ - m ( 1 + \alpha ) } }{ m ( t + i m \log p ) } \right. } \\
\quad \dsp{ - it \sum_{ m = 1 }^\infty \frac{ \chi_j( -1 ) e^{ - \alpha m \pi i } }{ m ( t - m \pi ) } + i \log \left( \frac{ \chi_j( -1 ) \Gamma( 1 + \alpha ) N_j^\alpha G( \bar{ \chi_j } ) }{ ( 2 \pi )^{ 1 + \alpha } } \right) - \frac{ ( 1 + \alpha ) \pi }{ 2 } } \\
\quad \dsp{ \left. + \frac{ 1 }{ t } \left( \gamma + \log \left( \frac{ 2 \pi }{ N_j } \right) + \frac{ \pi i }{ 2 } \right) - H ( t ) \right) - \frac{ t }{ 2 \pi } e^{ \frac{ it }{ 2 } } \int_{ S_j ( \pi \to 0 ) } e^{ - i s t } \log L ( s , \bar{ \chi_j } ) ds } \\
\quad \dsp{ - \frac{ i e^{ - \frac{ \chi_j ( -1 ) }{ 2 } i t } }{ 2 \sin t } - \mutaujz ( e^{ i \mutaujz t } + e^{ - i \mutaujz t } ) - \mujz }.
\end{array}
\right\} \label{Thm3.2-1}
\end{align}
{\upshape (iii)} $ l_{ \chi_j } ( t ) $ has the following approximate behavior at $ t = 0 $ :
\begin{align*}
l_{ \chi_j } ( t ) = - \frac{ \log t }{ 2 \pi t } - \frac{ 1 }{ 2 \pi t } \left( \log \left( \frac{ 2 \pi }{ N_j } \right) + \gamma + \frac{ 3 \pi i }{ 2 } \right) + O ( 1 ).
\end{align*}
{\upshape (iv)} $ l_{ \chi_j } ( t ) $ has simple poles in $ \mathbb{ C } - i \mathbb{ R }_{ \le 0 } $ only at the following points : 
\begin{align*}
\begin{cases}
t = im \log p, \\
t = - m \pi,
\end{cases}
\end{align*}
where $ m \in \mathbb{ Z }_{ \ge 1 } $.

In {\upshape (ii)} - {\upshape (iv)}, the argument lies in $ \left( - \frac{ \pi }{ 2 } , \frac{ 3 \pi }{ 2 } \right) $.
\end{thm}
\begin{proof}[\underline{\textbf{Proof of Theorem \ref{Thm3.2}}}]
(i) If $ \Re ( t ) > 0 $, then by Cauchy's theorem
\begin{align}
l_{ \chi_j } ( t ) = \frac{ 1 }{ 2 \pi t } \int_{ R_1 \cup S_j \cup R_2 } e^{ i \left( s - \frac{ 1 }{ 2 } \right) t } \frac{ L^\prime }{ L } ( s , \chi_j ) ds, \label{PofThm3.2-01}
\end{align}
where
\begin{align*}
R_1 & \coloneqq \{ s \in \mathbb{ C } \ | \ \Re ( s ) = - \alpha , \ \Im ( s ) \ge 0 \}, \\
R_2 & \coloneqq \{ s \in \mathbb{ C } \ | \ \Re ( s ) = 1 , \ \Im ( s ) \ge 0 \},
\end{align*}
and we go around the integral path in the counterclockwise direction. By the partial integration, $ ( \ref{PofThm3.2-01} ) $ becomes
\begin{align}
l_{ \chi_j } ( t ) = - \frac{ t }{ 2 \pi } e^{ - \frac{ i t }{ 2 } } \left( \dsp{ \int_{ R_1 } + \int_{ S_j } + \int_{ R_2 } } \right) e^{ i s t } \log L ( s , \chi_j ) ds. \label{PofThm3.2-02}
\end{align}
By using the functional equation
\begin{align*}
L ( s , \chi_j ) = \frac{ N_j^{ - s } }{ ( 2 \pi )^{ 1 - s } } G ( \chi_j ) \Gamma ( 1 - s ) ( e^{ - \frac{ \pi i }{ 2 } ( 1 - s ) } + \chi_j ( - 1 ) e^{ \frac{ \pi i }{ 2 } ( 1 - s ) } ) L ( 1 - s , \bar{ \chi_j } ),
\end{align*}
the integral of the path $ R_1 $ becomes
\begin{align}
\int_{ R_1 } & = \int_\infty^0 e^{ i ( - \alpha  + i y ) t } \log L ( - \alpha + i y , \chi_j ) i dy \notag \\
& = i e^{ - i \alpha t } \int_\infty^0 e^{ - y t } \left( \log \left( \frac{ N_j^\alpha G ( \chi_j ) }{ ( 2 \pi )^{ 1 + \alpha } } \right) \right. \label{PofThm3.2-03} \\
& \phantom{ i e^{ - i \alpha t } \int_\infty^0 e^{ - y t } } + i y \log \left( \frac{ 2 \pi }{ N_j } \right) \label{PofThm3.2-04} \\
& \phantom{ i e^{ - i \alpha t } \int_\infty^0 e^{ - y t } } + \log \Gamma ( 1 + \alpha - i y ) \label{PofThm3.2-05} \\
& \phantom{ i e^{ - i \alpha t } \int_\infty^0 e^{ - y t } } + \log ( e^{ - \frac{ \pi i }{ 2 } ( 1 + \alpha - iy ) } + \chi_j ( - 1 ) e^{ \frac{ \pi i }{ 2 } ( 1 + \alpha - i y ) } ) \label{PofThm3.2-06} \\
& \left. \phantom{ i e^{ - i \alpha t } \int_\infty^0 e^{ - y t } } + \log L ( 1 + \alpha - i y , \bar{ \chi_j } )  \right) dy. \label{PofThm3.2-07}
\end{align}
The integrals concerning $ ( \ref{PofThm3.2-03} ) $ and $ ( \ref{PofThm3.2-04} ) $ become
\begin{align}
\int_\infty^0 e^{ - y t } \log \left( \frac{ N_j^\alpha G ( \chi_j ) }{ ( 2 \pi )^{ 1 + \alpha } } \right) dy = - \frac{ 1 }{ t } \log \left( \frac{ N_j^\alpha G ( \chi_j ) }{ ( 2 \pi )^{ 1 + \alpha } } \right) \label{PofThm3.2-08}
\end{align}
and
\begin{align}
i \left( \log \left( \frac{ 2 \pi }{ N_j } \right) \right) \int_\infty^0 y e^{ - y t } dy = - \frac{ i }{ t^2 } \log \left( \frac{ 2 \pi }{ N_j } \right). \label{PofThm3.2-09}
\end{align}
respectively. By the partial integration the integral concerning $ ( \ref{PofThm3.2-05} ) $ is equal to
\begin{align}
& \int_\infty^0 e^{ - y t } \log \Gamma ( 1 + \alpha - i y ) dy \notag \\
& = \left[ - \frac{ 1 }{ t } e^{ - y t } \log \Gamma ( 1 + \alpha - i y ) \right]_\infty^0 - \frac{ i }{ t } \int_\infty^0 e^{ - y t } \frac{ \Gamma^\prime }{ \Gamma } ( 1 + \alpha - i y ) dy \notag \\
& = - \frac{ 1 }{ t } \log \Gamma ( 1 + \alpha ) + \frac{ i \gamma }{ t } \int_\infty^0 e^{ - y t } dy - \frac{ i }{ t } \int_\infty^0 e^{ - y t } \int_0^\infty \frac{ 1 - e^{ ( - \alpha + i y ) u } }{ e^u - 1 } du dy \label{PofThm3.2-10}
\end{align}
because
\begin{align*}
\frac{ \Gamma^\prime }{ \Gamma } ( s ) = - \gamma + \int_0^\infty \frac{ 1 - e^{ u ( 1 - s ) } }{ e^u - 1 } du \ \ ( \Re ( s ) > 0 ).
\end{align*}
The integral in the third term of $ ( \ref{PofThm3.2-10} ) $ becomes
\begin{align*}
& \int_0^\infty \frac{ 1 }{ e^u - 1 } \int_\infty^0 ( e^{ - y t } - e^{ - \alpha u + ( - t + i u ) y } ) dy du \\
&= \int_0^\infty \frac{ 1 }{ e^u - 1 } \left( - \frac{ 1 }{ t } + \frac{ e^{ - \alpha u } }{ t - i u } \right) du \\
& = - \frac{ 1 }{ t } \dsp{ \int_0^\infty \frac{ 1 }{ e^u - 1 } \cdot \frac{ u + it ( 1 - e^{ - \alpha u } ) }{ u + i t } du }.
\end{align*}
Hence,
\begin{align}
( \ref{PofThm3.2-10} ) = - \frac{ 1 }{ t } \log \Gamma ( 1 + \alpha ) - \frac{ i \gamma }{ t^2 } + \frac{ i }{ t^2 } \dsp{ \int_0^\infty \frac{ 1 }{ e^u - 1 } \cdot \frac{ u + i t ( 1 - e^{ - \alpha u } ) }{ u + i t } du }. \label{PofThm3.2-11}
\end{align}
The integral concerning $ ( \ref{PofThm3.2-06} ) $ is equal to
\begin{align}
&\int_\infty^0 e^{ - y t } \log ( e^{ - \frac{ \pi i }{ 2 } ( 1 + \alpha - i y ) } + \chi_j ( - 1 ) e^{ \frac{ \pi i }{ 2 } ( 1 + \alpha - i y ) } ) dy \notag \\
\begin{split}
& = \int_\infty^0 e^{ - y t } \log \chi_j ( - 1 ) dy + \int_\infty^0 \frac{ \pi i }{ 2 } ( 1 + \alpha - i y ) e^{ - y t } dy \\
& \quad + \int_\infty^0 e^{ - y t } \log ( 1 + \chi_j ( -1 ) e^{ - \pi i ( 1 + \alpha - i y ) } ) dy. \label{PofThm3.2-12}
\end{split}
\end{align}
Since the third term of $ ( \ref{PofThm3.2-12} ) $ becomes
\begin{align*}
\sum_{ m = 1 }^\infty \frac{ ( - 1 )^{ m - 1 } }{ m } \chi_j ( - 1 )^m e^{ - i m \pi ( 1 + \alpha ) } \int_\infty^0 e^{ - ( t + m \pi ) y } dy = \sum_{ m = 1 }^\infty \frac{ \chi_j ( - 1 )^m e^{ - i \alpha m \pi } }{ m ( t + m \pi ) },
\end{align*}
we have
\begin{align}
( \ref{PofThm3.2-12} ) = - \frac{ 1 }{ t } \log \chi_j ( -1 ) - \frac{ ( 1 + \alpha ) \pi i }{ 2 t } - \frac{ \pi }{ t^2 } + \dsp{ \sum_{ m = 1 }^\infty \frac{ \chi_j ( - 1 )^m e^{ - i m \alpha } }{ m ( t + m \pi ) } }. \label{PofThm3.2-13}
\end{align}
The integral concerning $ ( \ref{PofThm3.2-07} ) $ becomes
\begin{align}
\int_\infty^0 e^{ - y t } \log L ( 1 + \alpha - i y , \bar{ \chi_j } ) dy & = \sum_{ p } \sum_{ m = 1 }^\infty \frac{ \bar{ \chi_j } }{ m } p^{ - m ( 1 + \alpha ) } \int_\infty^0 e^{ ( - t + i m \log p ) y } dy \notag \\
& = - \sum_{ p } \sum_{ m = 1 }^\infty \frac{ \bar{ \chi_j } ( p^m ) p^{ - m ( 1 + \alpha ) } }{ m ( t - i m \log p ) }. \label{PofThm3.2-14}
\end{align}
The integral of the path $ R_2 $ of $ ( \ref{PofThm3.2-02} ) $ becomes
\begin{align}
\int_0^\infty e^{ i ( 1 + i y ) t } \log L ( 1 + i y , \chi_j ) i dy & = i e^{ i t } \sum_{ p } \sum_{ m = 1 }^\infty \frac{ \chi_j ( p^m ) }{ m } p^{ - m } \int_0^\infty e^{ - ( t + i m \log p ) y } dy \notag \\
& = i e^{ i t } \sum_{ p } \sum_{ m = 1 }^\infty \frac{ \chi_j ( p^m ) p^{ - m } }{ m ( t + i m \log p ) }. \label{PofThm3.2-15}
\end{align}

Applying $ ( \ref{PofThm3.2-08} ) , \, ( \ref{PofThm3.2-09} ) , \, ( \ref{PofThm3.2-11} ) , \, ( \ref{PofThm3.2-13} ) , \, ( \ref{PofThm3.2-14} ) $ and $ ( \ref{PofThm3.2-15} ) $ to $ ( \ref{PofThm3.2-02} ) $, we obtain the desired result. \\
(ii) By Theorem \ref{Thm3.2} (i), we find that for $ \Re ( t ) < 0 $
\begin{align*}
& l_{ \bar{ \chi_j } } ( - t ) \\
& = - \frac{ i t }{ 2 \pi } e^{ - \frac{ i t }{ 2 } } \sum_{ p } \sum_{ m = 1 }^\infty \frac{ \bar{ \chi_j } ( p^m ) p^{ - m } }{ m ( t - i m \log p ) } + \frac{ e^{ i \left( \alpha + \frac{ 1 }{ 2 } \right) t } }{ 2 \pi } \left( it \sum_{ p } \sum_{ m = 1 }^\infty \frac{ \chi_j ( p^m ) p^{ - m ( 1 + \alpha ) } }{ m ( t + i m \log p ) } \right. \\
& \quad - i t \sum_{ m = 1 }^\infty \frac{ \chi_j ( - 1 )^m e^{ i \alpha m \pi } }{ m ( t - m \pi ) } + i \log \left( \frac{ \chi_j ( - 1 ) \Gamma ( 1 + \alpha ) N_j^\alpha G ( \bar{ \chi_j } ) }{ ( 2 \pi )^{ 1 + \alpha } } \right) - \frac{ ( 1 + \alpha ) \pi }{ 2 } \\
& \quad \left. + \frac{ 1 }{ t } \left( \gamma + \log \left( \frac{ 2 \pi }{ N_j } \right) + \frac{ \pi i }{ 2 } \right) - H ( t ) \right) + \frac{ t }{ 2 \pi } e^{ \frac{ i t }{ 2 } } \int_{ S ( \pi \to 0 ) } e^{ - i s t } \log L ( s , \bar{ \chi_j } ) ds.
\end{align*}

By using the equation for $ l_{ \chi_j } ( t ) $ deduced in Theorem \ref{Thm3.1}, we obtain $ ( \ref{Thm3.2-1} ) $ for $ \Re ( t ) < 0 $. Since the right-hand side of $ ( \ref{Thm3.2-1} ) $ is meromorphic for $ t \in \mathbb{ C } - i \mathbb{ R }_{ \le 0 } $ if the argument lies in $ \left( - \frac{ \pi }{ 2 } , \frac{ 3 \pi }{ 2 } \right) $, the proof of (ii) is completed.

In the following, let $ t \in \mathbb{ C } - i \mathbb{ R }_{ \le 0 } $ and the argument lie in$ \left( - \frac{ \pi }{ 2 } , \frac{ 3 \pi }{ 2 } \right) $. \\
(iii) By Theorem \ref{Thm3.2} (ii) and Lemma \ref{Lem2.4} (i), we find that
\begin{align*}
l_{ \chi_j } ( t ) & = - \frac{ 1 }{ 2 \pi t } \left( \gamma + \log \left( \frac{ 2 \pi }{ N_j } \right) + \frac{ \pi i }{ 2 } \right) - \frac{ \log t }{ 4 \pi \sin \frac{ t }{ 2 } } - \frac{ i e^{ - \frac{ \chi_j ( -1 ) }{ 2 } i t } }{ 2 \sin t } + O ( 1 ) \quad ( t \to 0 ) \\
& = - \frac{ \log t }{ 2 \pi t } - \frac{ 1 }{ 2 \pi t } \left( \gamma + \log \left( \frac{ 2 \pi }{ N_j } \right) + \frac{ 3 \pi i }{ 2 } \right) + O ( 1 ) \quad ( t \to 0 ).
\end{align*}
(iv)  By $ ( \ref{lcDef} ) $, we find trivially that $ l_{ \chi_j } ( t ) $ is holomorphic for $ \Re ( t ) > 0 $. From this and the expression obtained in Theorem \ref{Thm3.2} (ii), the desired result follows.
\end{proof}
We consider the bounds of $ l_{ \chi_j } ( t ) $ which is needed later.
\begin{lem} \label{Lem3.3}
{\upshape (i)} For $ \Re ( t ) \ge 1 $
\begin{align*}
l_{ \chi_j } ( t ) = O( e^{ - \varepsilon_j \Re ( t ) + \frac{ 1 }{ 2 } | \Im ( t ) | } ).
\end{align*}
{\upshape (ii)} For $ \Re ( t ) \le - 1 $
\begin{align*}
l_{ \chi_j } ( t ) = \frac{ e^{ - \frac{ \chi_j ( -1 ) }{ 2 } it } }{ e^{ it } - e^{ - it } } + O ( e^{ \varepsilon_j \Re ( t ) + \frac{ 1 }{ 2 } | \Im ( t ) | } + e^{ \taucjz | \Im ( t ) | } ).
\end{align*}
{\upshape (iii)} If $ t = \sigma + i U $ with $ U \ge 2 $ and $ - U \le \sigma \le U $, then
\begin{align*}
l_{ \chi_j } ( t ) = \frac{ it }{ 2 \pi } e^{ - \frac{ it }{ 2 } } \sum_{ \scriptsize{ \begin{array}{c}
p, m \\
p^m < e^{ 2U }
\end{array} } } \frac{ \bar{ \chi_j ( p^m ) p^{ -m } } }{ m ( t - im \log p ) } + O( U e^{ \left( \varepsilon_j + \frac{ 1 }{ 2 } \right) U } ).
\end{align*}
\end{lem}
\begin{proof}[\underline{\textbf{Proof of Lemma \ref{Lem3.3}}}]
(i) If $ \Re ( t ) \ge 1 $, then we have
\begin{align*}
l_{ \chi_j } ( t ) = \sum_{ \Re ( \taucj ) > \varepsilon_j } e^{ - \taucj t }
\end{align*}
from $ ( \ref{lcDef} ) $. Since
\begin{align*}
\left| \sum_{ \Re ( \taucj ) > \varepsilon_j } e^{ - \taucj t } \right| & \le \sum_{ \Re ( \taucj ) > \varepsilon_j } | e^{ - \taucj t } | = \sum_{ \Re ( \taucj ) > \varepsilon_j } e^{ - \Re ( \taucj ) \Re ( t ) + \Im ( \taucj ) \Im ( t ) } \\
& \le e^{ \frac{ 1 }{ 2 } | \Im ( t ) | } \sum_{ \Re ( \taucj ) > \varepsilon_j } e^{ - \Re ( \taucj ) \Re ( t ) } \\
& = e^{ \frac{ 1 }{ 2 } | \Im ( t ) | } e^{ - \varepsilon_j \Re ( t ) } \sum_{ \Re ( \taucj ) > \varepsilon_j } e^{ - ( \Re ( \taucj ) - \varepsilon_j ) \Re ( t ) } \\
& = O ( e^{ - \varepsilon_j \Re ( t ) + \frac{ 1 }{ 2 } | \Im ( t ) | } ),
\end{align*}
we obtain the desired result. \\
(ii) For $ \Re ( t ) \le - 1 $, we have
\begin{align}
& l_{ \chi_j } ( t ) \notag \\
& = - l_{ \bar{ \chi_j } } ( - t ) - \frac{ i e^{ - \frac{ \chi_j ( - 1 ) }{ 2 } i t } }{ 2 \sin t } - \mutaujz ( e^{ i \taucjz t } + e^{ - i \taucjz t } ) - \mujz \notag \\
& = - \sum_{ \Re ( \taucjbar ) > \varepsilon_j } e^{ \taucjbar t } + \frac{ e^{ - \frac{ \chi_j ( - 1 ) }{ 2 } i t } }{ e^{ i t } - e^{ - i t } } - \mutaujz ( e^{ i \taucjz t } + e^{ - i \taucjz t } ) - \mujz \label{PofLem3.3-2}
\end{align}
by Theorem \ref{Thm3.1}. Concerning the first and third term of the right-hand side of $ ( \ref{PofLem3.3-2} ) $, we have
\begin{align*}
\left| \sum_{ \Re ( \taucjbar ) > \varepsilon_j } e^{ \taucjbar t } \right| & \le \sum_{ \Re ( \taucjbar ) > \varepsilon_j } | e^{ \taucjbar t } | = \sum_{ \Re ( \taucjbar ) > \varepsilon_j } e^{ \Re ( \taucjbar ) \Re ( t ) - \Im ( \taucjbar ) \Im ( t ) } \\
& \le e^{ \frac{ 1 }{ 2 } | \Im ( t ) | } \sum_{ \Re ( \taucjbar ) > \varepsilon_j } e^{ \Re ( \taucjbar ) \Re ( t ) } \\
& = e^{ \frac{ 1 }{ 2 } | \Im ( t ) | } e^{ \varepsilon_j \Re ( t ) } \sum_{ \Re ( \taucjbar ) > \varepsilon_j } e^{ ( \Re ( \taucjbar ) - \varepsilon_j ) \Re ( t ) } \\
& = O ( e^{ \varepsilon_j \Re ( t ) + \frac{ 1 }{ 2 } | \Im ( t ) | } )
\end{align*}
and
\begin{align*}
| e^{ i \taucjz t } + e^{ - i \taucjz t } | \le e^{ - \taucjz \Im ( t ) } + e^{ \taucjz \Im ( t ) } = O ( e^{ \taucjz | \Im ( t ) | } )
\end{align*}
respectively. Hence, we obtain the desired result. \\
(iii) When $ t = \sigma + i U $ with $ U \ge 2 $ and $ - U \le \sigma \le U $, we have
\begin{align}
l_{ \chi_j } ( t ) = \frac{ i t }{ 2 \pi } e^{ - \frac{ i t }{ 2 } } \sum_{ p } \sum_{ m = 1 }^\infty \frac{ \bar{ \chi_j } ( p^m ) p^{ - m } }{ m ( t - i m \log p ) } + O ( U e^{ \left( \frac{ 1 }{ 2 } + \varepsilon_j \right) U } )
\end{align}
by estimating trivially each term of the right-hand side of $ ( \ref{Thm3.2-1} ) $ in Theorem \ref{Thm3.2} except the first term.

If $ p^m \ge e^{ 2 U } $, then $ U \le \frac{ m \log p }{ 2 } $, so $ m \log p - U \ge \frac{ m \log p }{ 2 } $. Therefore, we have
\begin{align*}
& \Bigg| \, ( \sigma + i U ) e^{ - \frac{ i }{ 2 } ( \sigma + i U ) } \sum_{ \scriptsize{ \begin{array}{c}
p , m \\
p^m \ge e^{ 2 U }
\end{array} } } \frac{ \bar{ \chi_j } ( p^m ) p^{ - m } }{ m ( \sigma + i U - i m \log p ) } \, \Bigg| \\
& \le 2 U e^{ \frac{ U }{ 2 } } \sum_{ \scriptsize{ \begin{array}{c}
p , m \\
p^m \ge e^{ 2 U }
\end{array} } } \frac{ p^{ - m } }{ m ( m \log p - U ) } \le 4 U e^{ \frac{ U }{ 2 } } \sum_{ \scriptsize{ \begin{array}{c}
p , m \\
p^m \ge e^{ 2 U }
\end{array} } } \frac{ p^{ - m } }{ m^2 \log p } \\
& \le 4 U e^{ \frac{ U }{ 2 } } \sum_{ p , m } \frac{ p^{ - m } }{ m^2 \log p } = O( U e^{ \frac{ U }{ 2 } } ).
\end{align*}
In the last equation, we use Lemma \ref{Lem2.7} (ii). This completes the proof.
\end{proof}
Now, we fix $ \theta_j $ arbitrarily with $ 0 < \theta_j < \frac{ \pi }{ 4 } $ and $ \tan \theta_j < \varepsilon_j $.
\begin{cor} \label{Cor3.4}
{\upshape (i)} For $ u \ge \frac{ 1 }{ \cos \theta_j } $
\begin{align}
&l_{ \chi_j } ( u e^{ - i \theta_j } ) = O( e^{ - \frac{ u }{ 2 } \sin \theta_j } ), \label{Cor3.4-1} \\
&l_{ \chi_j } ( ue^{ i ( \pi - \theta_j ) } ) = O( e^{ \taucjz u \sin \theta_j } ). \label{Cor3.4-2}
\end{align}
{\upshape (ii)} If $ R \ge 1 $ and $ - R \tan \theta_j \le y \le R \tan \theta_j $ then
\begin{align*}
l_{ \chi_j } ( R + i y ) = O( e^{ \frac{ y }{ 2 } } ).
\end{align*}
{\upshape (iii)} If $ \sigma \in \mathbb{ R } $, $ M \in \mathbb{ Z }_{ \ge 100 } $ and $ U \coloneqq \log \left( M + \frac{ 1 }{ 2 } \right)$ then
\begin{align*}
l_{ \chi_j } ( \sigma + i U ) = \begin{cases}
O( e^{ \frac{ U }{ 2 } } ) & ( \sigma \ge 1 ), \\
O( U^2 e^{ \left( \varepsilon_j + \frac{ 1 }{ 2 } \right) U } ) & ( - 1 \le \sigma \le 1 ), \\
O( e^{ \varepsilon_j \sigma + \frac{ U }{ 2 } } + e^{ \taucjz U } ) & ( \sigma \le -1 ).
\end{cases}
\end{align*}
\end{cor}
\begin{proof}[\underline{\textbf{Proof of Corollary \ref{Cor3.4}}}]
(i) First, by Lemma \ref{Lem3.3} (i) we find
\begin{align*}
l_{ \chi_j } ( u e^{ - i \theta_j } ) & = O ( e^{ - \varepsilon_j \Re ( u e^{ - i \theta_j } ) + \frac{ 1 }{ 2 } | \Im ( u e^{ - i \theta_j } ) | } ) \\
& = O ( e^{ - u ( \varepsilon_j \cos \theta_j - \frac{ 1 }{ 2 } \sin \theta_j ) } ) \\
& = O ( e^{ - \frac{ 1 }{ 2 } u \sin \theta_j } ).
\end{align*}
In the last equation we use the fact that $ \frac{ 1 }{ 2 } \sin \theta_j < \varepsilon_j \cos \theta_j - \frac{ 1 }{ 2 } \sin \theta_j $ because $ \tan \theta_j < \varepsilon_j $. Hence, $ ( \ref{Cor3.4-1} ) $ has been proved.

Next, by Lemma \ref{Lem3.3} (ii) we have
\begin{align}
\begin{split}
& l_{ \chi_j } ( u e^{ i ( \pi - \theta_j ) } ) = \frac{ e^{ - \frac{ \chi_j ( - 1 ) }{ 2 } i u e^{ i ( \pi - \theta_j ) } } }{ e^{ i u e^{ i ( \pi - \theta_j ) } } - e^{ - i u e^{ i ( \pi - \theta_j ) } } } \\
& \phantom{ l_{ \chi_j } ( u e^{ i ( \pi - \theta_j ) } ) = } + O ( e^{ \varepsilon_j \Re ( u e^{ i ( \pi - \theta_j ) } ) + \frac{ 1 }{ 2 } | \Im ( u e^{ i ( \pi - \theta_j ) } ) | } + e^{ \taucjz | \Im ( u e^{ i ( \pi - \theta_j ) } ) | } ).
\end{split} \label{PofCor3.4-1}
\end{align}
Now,
\begin{align*}
& \left| \, ( \mbox{the first term of the right-hand side of } ( \ref{PofCor3.4-1} ) ) \, \right| \\
& \le \frac{ e^{ \frac{ 1 }{ 2 } u \sin \theta_j } }{ e^{ u \sin \theta_j } - e^{ - u \sin \theta_j } } = O ( e^{ - \frac{ 1 }{ 2 } u \sin \theta_j } ).
\end{align*}
Hence, we can deduce
\begin{align*}
l_{ \chi_j } ( u e^{ i ( \pi - \theta_j ) } ) & = O ( e^{ - \frac{ 1 }{ 2 } u \sin \theta_j } + e^{ - u ( \varepsilon_j \cos \theta_j - \frac{ 1 }{ 2 } \sin \theta_j ) } + e^{ \taucjz u \sin \theta_j } ) \\
& = O ( e^{ \taucjz u \sin \theta_j } ),
\end{align*}
where in the last equation we use the fact that $ \varepsilon_j \cos \theta_j - \frac{ 1 }{ 2 } \sin \theta_j >0 $ because $ \tan \theta_j < \varepsilon_j $. Hence, $ ( \ref{Cor3.4-2} ) $ holds. \\
(ii) From Lemma \ref{Lem3.3} (i) and $ \tan \theta_j < \varepsilon_j $, we can easily deduce the desired result. \\
(iii) If $ \sigma \ge 1 \, ( \mbox{respectively } \sigma \le - 1 ) $ then we can trivially deduce the desired result from Lemma \ref{Lem3.3} (i) (respectively Lemma \ref{Lem3.3} (ii)).

If $ - 1 \le \sigma \le 1 $ then we can derive
\begin{align*}
& l_{ \chi_j } ( \sigma + i U ) \\
& = \frac{ i ( \sigma + i U ) }{ 2 \pi } e^{ - \frac{ i }{ 2 } ( \sigma + i U ) } \sum_{ \scriptsize{ \begin{array}{c}
p , m \\
p^m < e^{ 2 U }
\end{array} } } \frac{ \bar{ \chi_j } ( p^m ) p^{ - m } }{ m ( \sigma + i U - i m \log p ) } + O ( U e^{ \left( \varepsilon_j + \frac{ 1 }{ 2 } \right) U } )
\end{align*}
from Lemma \ref{Lem3.3} (iii). Concerning the first term of the right-hand side, we find that
\begin{align}
& \left| ( \sigma + i U ) e^{ - \frac{ i }{ 2 } ( \sigma + i U ) } \right| = O ( U e^{ \frac{ U }{ 2 } } ), \notag \\
& \Bigg| \sum_{ \scriptsize{ \begin{array}{c}
p , m \\
p^m < e^{ 2 U }
\end{array} } } \frac{ \bar{ \chi_j } ( p^m ) p^{ - m } }{ m ( \sigma + i U - i m \log p ) } \, \Bigg| \le \sum_{ \scriptsize{ \begin{array}{c}
p , m \\
p^m < e^{ 2 U }
\end{array} } } \frac{ p^{ - m } }{ | \ U - m \log p \ | } \notag \\
& = \sum_{ \scriptsize{ \begin{array}{c}
p , m \\
p^m < M + \frac{ 1 }{ 2 }
\end{array} } } \frac{ p^{ - m } }{ U - m \log p } + \sum_{ \scriptsize{ \begin{array}{c}
p , m \\
M + \frac{ 1 }{ 2 } \le p^m < \left( M + \frac{ 1 }{ 2 } \right)^2
\end{array} } } \frac{ p^{ - m } }{ m \log p - U } \label{PofCor3.4-3}
\end{align}
and that by Lemma \ref{Lem2.7} (i)
\begin{alignat*}{2}
& ( \mbox{the first term of } ( \ref{PofCor3.4-3} ) ) && \le \left( M + \frac{ 1 }{ 2 } \right) \sum_{ \scriptsize{ \begin{array}{c}
p , m \\
p^m < M + \frac{ 1 }{ 2 }
\end{array} } } \frac{ p^{ - m } }{ M + \frac{ 1 }{ 2 } - p^m } \\
&&& \le \left( M + \frac{ 1 }{ 2 } \right) \sum_{ n = 2 }^M \frac{ 1 }{ n \left( M + \frac{ 1 }{ 2 } - n \right) } \\
&&& = \sum_{ n = 2 }^M \frac{ 1 }{ n } + \sum_{ n = 2 }^M \frac{ 1 }{ M + \frac{ 1 }{ 2 } - n } \\
&&& \ll \log M \ll U, \\
& ( \mbox{the second term of } ( \ref{PofCor3.4-3} ) ) && \le \sum_{ \scriptsize{ \begin{array}{c}
p , m \\
M + \frac{ 1 }{ 2 } \le p^m < \left( M + \frac{ 1 }{ 2 } \right)^2
\end{array} } } \frac{ p^{ - m } }{ m \log p - U } \\
&&& \le \sum_{ \scriptsize{ \begin{array}{c}
p , m \\
M + \frac{ 1 }{ 2 } \le p^m < \left( M + \frac{ 1 }{ 2 } \right)^2
\end{array} } } \frac{ 1 }{ p^m - \left( M + \frac{ 1 }{ 2 } \right) } \\
&&& \le \sum_{ n = 1 }^{ M^2 } \frac{ 1 }{ ( n + M ) - ( M + \frac{ 1 }{ 2 } ) } \ll U.
\end{alignat*}
Hence, we can obtain
\begin{align*}
l_{ \chi_j } ( \sigma + i U ) = O ( U^2 e^{ \frac{ U }{ 2 } } + U e^{ ( \varepsilon_j + \frac{ 1 }{ 2 } ) U } ) = O ( U^2 e^{ ( \varepsilon_j + \frac{ 1 }{ 2 } ) U } ).
\end{align*}
This completes the proof.
\end{proof}
\section{\textbf{The ``key equation''}} \label{Sec4}
In this section, we prove an equation we name the ``key equation'' which links the ``factors series'' of $ ( L_{ \chi_1 } \underset{ \mathbb{ F }_1 }{ \otimes } \cdots \underset{ \mathbb{ F }_1 }{ \otimes } L_{ \chi_r } ) ( s ) $ to $ r $-tuples of prime numbers $ ( r \in \mathbb{ Z }_{ \ge 1 } ) $.

Define that
\begin{align*}
\theta^{ ( r ) } \coloneqq \underset{ j \in \{ 1 , \cdots , r \} }{ \min } \{ \theta_j \},
\end{align*}
\begin{align*}
\taurz \coloneqq \underset{ j \in \{ 1 , \cdots , r \} }{ \max } \{ \taucjz \},
\end{align*}
\begin{align*}
D_{ \theta^{ ( r ) } , \taurz } &\coloneqq \left\{ ( w , z ) \in \mathbb{ C }^2 \ \left| \ - \frac{ r }{ 2 } \sin \theta^{ ( r ) } < \Re ( z e^{ - i \theta^{ ( r ) } } ) < - r \taurz \sin \theta^{ ( r ) } \right. \right\} \\
&= \left\{ ( w , z ) \in \mathrm{ C }^2 \ \left| \ - \frac{ r }{ 2 } \tan \theta^{ ( r ) } < \Re ( z ) + \Im ( z ) \tan \theta^{ ( r ) } < - r \taurz \tan \theta^{ ( r ) } \right. \right\},
\end{align*}
\begin{align}
& L_{ \theta^{ ( r ) } }^{ (1) } ( w , z , \{ \chi_j \}_{ j = 1 }^r ) \coloneqq \frac{ 1 }{ \Gamma ( w ) } \int_0^{ \infty e^{ - i \theta^{ ( r ) } } } e^{ - zt } \prod_{ j = 1 }^r l_{ \chi_j } ( t ) t^{ w - 1 } dt, \label{L1Def} \\
& L_{ \theta^{ ( r ) } }^{ ( 2 ) } ( w , z , \{ \chi_j \}_{ j = 1 }^r ) \notag \\
\begin{split}
& \coloneqq ( -1 )^{ r - 1 } \frac{ e^{ \pi i w } }{ \Gamma ( w ) } \int_0^{ \infty e^{ - i \theta^{ ( r ) } } } e^{ zt } \prod_{ j = 1 }^r \left( l_{ \bar{ \chi_j } } ( t ) + \sum_{ n = 1 }^\infty e^{ - \left( 2n - 1 - \frac{ \chi_j ( -1 ) }{ 2 } \right) it } \right. \\
& \left. \phantom{ \sum_{ n = 1 }^\infty e^{ - \left( 2n - 1 - \frac{ \chi_j ( -1 ) }{ 2 } \right) it } } + \mutaujz e^{ i \taucjz t } + \mutaujz e^{ - i \taucjz t } + \mujz \right) t^{ w - 1 } dt , 
\end{split} \label{L2Def} \\
& R_{ \theta^{ ( r ) } } ( w , z , \{ \chi_j \}_{ j = 1 }^r ) \coloneqq \frac{ 2 \pi i }{ \Gamma ( w ) } \lim_{ N \to \infty } \sum_{ \scriptsize{ \begin{array}{c}
p , m \\
p^m < N + 1 / 2
\end{array} } } \underset{ t = im \log p }{ \mathrm{ Res } } e^{ - zt } \prod_{ j = 1 }^r l_{ \chi_j } ( t ) t^{ w - 1 }. \notag
\end{align}
Then, we show
\begin{thm}[\textbf{The ``key equation''}] \label{Thm4.1}
Let $ ( w , z ) \in D_{ \theta^{ ( r ) } , \taurz } $ satisfy $ \Im ( z ) < - ( \frac{ 1 }{ 2 } + \varepsilon^{ ( r ) } ) r $ and $ \mathrm{ Re ( w ) } > r $. Then,
\begin{align}
L_{ \theta^{ ( r ) } }^{ ( 1 ) } ( w , z , \{ \chi_j \}_{ j = 1 }^r ) + L_{ \theta^{ ( r ) } }^{ ( 2 ) } ( w , z , \{ \chi_j \}_{ j = 1 }^r ) = R_{ \theta^{ ( r ) } } ( w , z , \{ \chi \}_{ j = 1 }^r ). \label{Thm4.1-1}
\end{align}
\end{thm}
\begin{proof}[\underline{\textbf{Proof of Theorem \ref{Thm4.1}}}]
Let $ \lambda $ be any fixed real number with $ 0 < \lambda < \log 2 $ and we define
\begin{align*}
F_{ \theta^{ ( r ) } } ( w , z , \{ \chi_j \}_{ j = 1 }^r ; \lambda ) \coloneqq \frac{ 1 }{ \Gamma ( w ) } \int_{ V_{ \lambda , \theta^{ ( r ) } } } e^{ - z t } \prod_{ j = 1 }^r l_{ \chi_j } ( t ) t^{ w - 1 } dt,
\end{align*}
where $ V_{ \lambda , \theta^{ ( r ) } } $ is the union of $ V_1 ( \infty \to \lambda ) $, $ V_2 ( \pi - \theta^{ ( r ) } \to - \theta^{ ( r ) } ) $ and $ V_3 ( \lambda \to \infty ) $ when
\begin{align*}
V_1 & \coloneqq \{ \nu e^{ i ( \pi - \theta^{ ( r ) } ) } \: | \: \nu \ge \lambda \}, \\
V_2 & \coloneqq \{ \lambda e^{ i \varphi } \: | \: - \theta^{ ( r ) } \le \varphi \le \pi - \theta^{ ( r ) } \}, \\
V_3 & \coloneqq \{ \nu e^{ - i \theta^{ ( r ) } } \: | \: \nu \ge \lambda \}.
\end{align*}

By Corollary \ref{Cor3.4} (i), for large enough $ u $
\begin{align*}
l_{ \chi_j } ( u e^{ - i \theta^{ ( r ) } } ) = O ( e^{ - \frac{ u }{ 2 } \sin \theta^{ ( r ) } } ) , \, l_{ \chi_j } ( u e^{ i ( \pi - \theta^{ ( r ) } ) } ) = O ( e^{ \taurz u \sin \theta^{ ( r ) } } ).
\end{align*}
Therefore, $ F_{ \theta^{ ( r ) } } ( w , z , \{ \chi_j \}_{ j = 1 }^r ; \lambda ) $ converges absolutely and uniformly on any compact subset of $ D_{ \theta^{ ( r ) } , \taurz } $.

Now, when $ ( w , z ) \in D_{ \theta^{ ( r ) } , \taurz } $ and $ 0 < \eta < \lambda $, we have
\begin{align*}
& F_{ \theta^{ ( r ) } } ( w , z , \{ \chi_j \}_{ j = 0 }^r ; \lambda ) - F_{ \theta^{ ( r ) } } ( w , z , \{ \chi_j \}_{ j = 1 }^r ; \eta ) \\
&= \frac{ 1 }{ \Gamma ( w ) } \int_{ W_{ \eta , \lambda , \theta^{ ( r ) } } } e^{ - z t } \prod_{ j = 1 }^r l_{ \chi_j } ( t ) t^{ w - 1 } dt = 0
\end{align*}
by Theorem \ref{Thm3.2} (iv) and Cauchy's theorem, where
\begin{align*}
W_{ \eta , \lambda , \theta^{ ( r ) } } \coloneqq & \{ \lambda e^{ i \varphi } \: | \: - \theta^{ ( r ) } \le \varphi \le \pi - \theta^{ ( r ) } \} \cup \{ \nu e^{ - i \theta^{ ( r ) } } \: | \: \eta \le \nu \le \lambda \} \\
& \cup \{ \eta e^{  i \varphi } \: | \: - \theta^{ ( r ) } \le \varphi \le \pi - \theta^{ ( r ) } \} \cup \{ R e^{ i ( \pi - \theta^{ ( r ) } ) } \: | \: \eta \le \nu \le \lambda \}
\end{align*}
and we go around the integral path in the counterclockwise direction. If $ \Re ( w ) > r $, then by Theorem \ref{Thm3.2} (iii) we have
\begin{align}
& F_{ \theta^{ ( r ) } } ( w , z , \{ \chi_j \}_{ j = 1 }^r ; \lambda ) = \lim_{ \eta \downarrow 0 } F_{ \theta^{ ( r ) } } ( w , z, \{ \chi_j \}_{ j = 1 }^r ; \eta ) \notag \\
\begin{split}
& = \frac{ 1 }{ \Gamma ( w ) } \int_{ \infty e^{ i ( \pi - \theta^{ ( r ) } ) } }^0 e^{ - z t } \prod_{ j = 1 }^r l_{ \chi_j } ( t ) t^{ w - 1 } dt \\
& \quad + \frac{ 1 }{ \Gamma ( w ) } \int_0^{ \infty e^{ - i \theta^{ ( r ) } } } e^{ - z t } \prod_{ j = 1 }^r l_{ \chi_j } ( t ) t^{ w - 1 } dt.
\end{split} \label{PofThm4.1-1}
\end{align}
By replacing $ t $ with $ - t $, using Theorem \ref{Thm3.1} and taking note of
\begin{align*}
\frac{ i e^{ \frac{ \chi_j ( - 1 ) }{ 2 } it } }{ 2 \sin t } = - \frac{ e^{ \left( \frac{ \chi_j ( - 1 ) }{ 2 } - 1 \right) i t } }{ 1 - e^{ - 2 i t } } = - \sum_{ n = 1 }^\infty e^{ - \left( 2 n - 1 - \frac{ \chi_j ( - 1 ) }{ 2 } \right) i t },
\end{align*}
we find that the first term of $ ( \ref{PofThm4.1-1} ) $ is equal to
\begin{align*}
& ( - 1 )^{ r - 1 } \frac{ e^{ \frac{ \pi i w }{ 2 } } }{ \Gamma ( w ) } \int_0^{ \infty e^{ - i \theta^{ ( r ) } } } e^{ z t } \prod_{ j = 1 }^r \left( l_{ \bar{ \chi_j } } ( t ) + \sum_{ n = 1 }^\infty e^{ - \left( 2 n - 1 - \frac{ \chi_j ( - 1 ) }{ 2 } \right) i t } \right. \\
& \left. \phantom{ \sum_{ n = 1 }^\infty e^{ - \left( 2 n - 1 - \frac{ \chi_j ( - 1 ) }{ 2 } \right) i t } } + \mutaujz e^{ i \taucjz t } + \mutaujz e^{ - i \taucjz t } + \mujz \right) t^{ w - 1 } dt.
\end{align*}
Hence, we have
\begin{align*}
F_{ \theta^{ ( r ) } } ( w , z , \{ \chi_j \}_{ j = 1 }^r ; \lambda ) = L_{ \theta^{ ( r ) } }^{ ( 1 ) } ( w , z , \{ \chi_j \}_{ j = 1 }^r ) + L_{ \theta^{ ( r ) } }^{ ( 2 ) } ( w , z , \{ \chi_j \}_{ j = 1 }^r ).
\end{align*}

Next, we define that $ U \coloneqq \log \left( M + \frac{ 1 }{ 2 } \right) $ for $ M \in \mathbb{ Z }_{ \ge 100 } $ and let $ ( w , z ) \in D_{ \theta^{ ( r ) } , \taurz } $ with $ \Im ( z ) < - \left( \frac{ 1 }{ 2 } + \varepsilon^{ ( r ) } \right) r $ and $ R \in \mathbb{ R } $ with $ R \tan \theta^{ ( r ) } \ge U $. By Theorem \ref{Thm3.2} (iv) and the residue theorem, we have
\begin{align}
\int_{ P_1 \cup P_2 \cup P_3 } e^{ - z t } \prod_{ j = 1 }^r l_{ \chi_j } ( t ) t^{ w - 1 } dt = 2 \pi i \sum_{ \scriptsize{ \begin{array}{c}
p , m \\
p^m < M + \frac{ 1 }{ 2 }
\end{array} } } \underset{ t = i m \log p }{ \mathrm{ Res } } e^{ - z t } \prod_{ j = 1 }^r l_{ \chi_j } ( t ) t^{ w - 1 }, \label{PofThm4.1-2}
\end{align}
where
\begin{align*}
P_1 & \coloneqq \left\{ - u + i u \tan \theta^{ ( r ) } \ \left| \ \frac{ U }{ \tan \theta^{ ( r ) } } \le u \le \lambda \cos \theta^{ ( r ) } \right. \right\} \\
& \quad \cup \{ \lambda e^{ i \varphi } \: | \: - \theta^{ ( r ) } \le \varphi \le \pi - \theta^{ ( r ) } \} \\
& \quad \cup \{ u - i u \tan \theta^{ ( r ) } \: | \: \lambda \cos \theta^{ ( r ) } \le u \le R \}, \\
P_2 & \coloneqq \{ R + i y \: | \: - R \tan \theta^{ ( r ) } \le y \le U \}, \\
P_3 & \coloneqq \left\{ \sigma + i U \: \left| \: - \frac{ U }{ \tan \theta^{ ( r ) } } \le \sigma \le R  \right. \right\}
\end{align*}
and we go around the integral path in the counterclockwise direction. First, we consider the limit of $ ( \ref{PofThm4.1-2} ) $ as $ R \to \infty $. Concerning the integral of the path $ P_2 $, we have
\begin{align}
& \left| \int_{ - R \tan \theta^{ ( r ) } }^U e^{ - z ( R + i y ) } \prod_{ j = 1 }^r l_{ \chi_j } ( R + i y ) ( R + i y )^{ w - 1 } i dy \right| \notag \\
& \ll_{ r , w } R^{ \Re ( w ) - 1 } e^{ - \Re ( z ) R } \int_{ - R \tan \theta^{ ( r ) } }^U e^{ ( \Im ( z ) + \frac{ r }{ 2 } ) y } dy \notag \\
& \le R^{ \Re ( w ) - 1 } e^{ - \Re ( z ) R } \int_{ - R \tan \theta^{ ( r ) } }^{ R \tan \theta^{ ( r ) } } e^{ ( \Im ( z ) + \frac{ r }{ 2 } ) y } dy \notag \\
& = R^{ \Re ( w ) - 1 } e^{ - \Re ( z ) R } \frac{ e^{ ( \Im ( z ) + \frac{ r }{ 2 } ) R \tan \theta^{ ( r ) } } - e^{ - ( \Im ( z ) + \frac{ r }{ 2 } ) R \tan \theta^{ ( r ) } } }{ \Im ( z ) + \frac{ r }{ 2 } } \notag \\
& \le - \frac{ R^{ \Re ( w ) - 1 } }{ \Im ( z ) + \frac{ r }{ 2 } } e^{ - ( \Re ( z ) + \Im ( z ) \tan \theta^{ ( r ) } + \frac{ r }{ 2 } \tan \theta^{ ( r ) } ) R }, \label{PofThm4.1-3}
\end{align}
where in the last inequality we use the fact that $ \Im ( z ) + \frac{ r }{ 2 } < 0 $. From $ \Re ( z ) + \Im ( z ) \tan \theta + \frac{ r }{ 2 } \tan \theta > 0 $ because $ ( w , z ) \in D_{ \theta^{ ( r ) } , \taurz } $, it follows that $ ( \ref{PofThm4.1-3} ) $ vanishes as $ R \to \infty $. Hence, we have
\begin{align}
\int_{ P_4 \cup P_5 } e^{ - z t } \prod_{ j = 1 }^r l_{ \chi_j } ( t ) t^{ w - 1 } dt = 2 \pi i \sum_{ \scriptsize{ \begin{array}{c}
p , m \\
p^m < M + \frac{ 1 }{ 2 }
\end{array} } } \underset{ t = im \log p }{ \mathrm{ Res } } e^{ - z t } \prod_{ j = 1 }^r l_{ \chi_j } ( t ) t^{ w - 1 }, \label{PofThm4.1-4}
\end{align}
where
\begin{align*}
P_4 \coloneqq & \left\{ - u + i u \tan \theta^{ ( r ) } \: \left| \: \frac{ U }{ \tan \theta^{ ( r ) } } \le u \le \lambda \cos \theta^{ ( r ) } \right. \right\} \\
& \cup \{ \lambda e^{ i \varphi } \ | \ - \theta^{ ( r ) } \le \varphi \le \pi - \theta^{ ( r ) } \} \\
& \cup \{ u - i u \tan \theta^{ ( r ) } \: | \: u \ge \lambda \cos \theta^{ ( r ) } \}, \\
P_5 \coloneqq & \left\{ \sigma + i U \: \left| \: \sigma \ge - \frac{ U }{ \tan \theta^{ ( r ) } } \right. \right\}
\end{align*}
and we go around the integral path in the counterclockwise direction. Next, we consider the limit of $ ( \ref{PofThm4.1-4} ) $ as $ M \to \infty $. Concerning the integral of the path $ P_5 $, we have
\begin{align}
\left| \int_{ P_5 } \right| & = \left| \int_\infty^{ - \frac{ U }{ \tan \theta^{ ( r ) } } } e^{ - z ( \sigma + i U ) } \prod_{ j = 1 }^r l_{ \chi_j } ( \sigma + i U ) ( \sigma + i U )^{ w - 1 } d \sigma \right| \notag \\
& \ll_w \int_{ - \frac{ U }{ \tan \theta^{ ( r ) } } }^\infty e^{ - \Re ( z ) \sigma + \Im ( z ) U } \left| \prod_{ j = 1 }^r l_{ \chi_j } ( \sigma + i U ) \right| \max \{ | \sigma | , U \}^{ \Re ( w ) - 1 } d \sigma \notag \\
& = \dsp{ \int_{ - \frac{ U }{ \tan \theta^{ ( r ) } } }^{ - 1 } } + \dsp{ \int_{ - 1 }^1 } + \dsp{ \int_1^\infty }. \label{PofThm4.1-5}
\end{align}
About the first term of $ ( \ref{PofThm4.1-5} ) $, by using Corollary \ref{Cor3.4} (iii) we can deduce
\begin{align}
& \int_{ - \frac{ U }{ \tan \theta^{ ( r ) } } }^{ - 1 } \notag \\
& \ll_r \int_{ - \frac{ U }{ \tan \theta^{ ( r ) } } }^{ - U } e^{ - \Re ( z ) \sigma + \Im ( z ) U } ( e^{ ( \varepsilon^{ ( r ) } \sigma + \frac{ U }{ 2 } ) r } + e^{ \taurz U r } ) ( - \sigma )^{ \Re ( w ) - 1 } d \sigma \notag \\
& \quad + \int_{ - U }^{ - 1 } e^{ - \Re ( z ) \sigma + \Im ( z ) U } ( e^{ ( \varepsilon^{ ( r ) } \sigma + \frac{ U }{ 2 } ) r } + e^{ \taurz U r } ) U^{ \Re ( w ) - 1 } d \sigma \notag \\
& \le \int_{ - \frac{ U }{ \tan \theta^{ ( r ) } } }^{ - 1 } e^{ - \Re ( z ) \sigma + \Im ( z ) U } ( e^{ ( \varepsilon^{ ( r ) } \sigma + \frac{ U }{ 2 } ) r } + e^{ \taurz U r } ) ( - \sigma )^{ \Re ( w ) - 1 } d \sigma \notag \\
& \quad + \int_{ - \frac{ U }{ \tan \theta^{ ( r ) } } }^{ - 1 } e^{ - \Re ( z ) \sigma + \Im ( z ) U } ( e^{ ( \varepsilon^{ ( r ) } \sigma + \frac{ U }{ 2 } ) r } + e^{ \taurz U r } ) U^{ \Re ( w ) - 1 } d \sigma \notag \\
& \le e^{ \Im ( z ) U } \left( \frac{ U }{ \tan \theta^{ ( r ) } } \right)^{ \Re ( w ) - 1 } \notag \\
& \qquad \times \int_{ - \frac{ U }{ \tan \theta^{ ( r ) } } }^{ - 1 } ( e^{ ( \varepsilon^{ ( r ) } r - \Re ( z ) ) \sigma + \frac{ U }{ 2 } r } + e^{ - \Re ( z ) \sigma + \taurz U r } ) d \sigma \notag \\
& \quad + e^{ \Im ( z ) U } U^{ \Re ( w ) - 1 } \int_{ - \frac{ U }{ \tan \theta^{ ( r ) } } }^{ - 1 } ( e^{ ( \varepsilon^{ ( r ) } r - \Re ( z ) ) \sigma + \frac{ U }{ 2 } r } + e^{ - \Re ( z ) \sigma + \taurz U r } ) d \sigma \notag \\
\begin{split}
\phantom{\int_{ - \frac{ U }{ \tan \theta^{ ( r ) } } }^{ - 1 } } &= e^{ \Im ( z ) U } \left( U^{ \Re ( w ) - 1 } + \left( \frac{ U }{ \tan \theta^{ ( r ) } } \right)^{ \Re ( w ) - 1 } \right) \\
& \quad \times \int_{ - \frac{ U }{ \tan \theta^{ ( r ) } } }^{ - 1 } ( e^{ ( \varepsilon^{ ( r ) } r - \Re ( z ) ) \sigma + \frac{ U }{ 2 } r } + e^{ - \Re ( z ) \sigma + \taurz U r } ) d \sigma.
\end{split} \label{PofThm4.1-6}
\end{align}
Since
\begin{align*}
\int_{ - \frac{ U }{ \tan \theta^{ ( r ) } } }^{ - 1 } e^{ A \sigma } d \sigma \ll_A \left\{ \begin{array}{@{}ll}
1 & ( A > 0 ), \\
\dsp{ \frac{ U }{ \tan \theta^{ ( r ) } } } & ( A = 0 ), \\[+10pt]
e^{ - A \frac{ U }{ \tan \theta^{ ( r ) } } } & ( A < 0 )
\end{array} \right\} \ll \frac{ U }{ \tan \theta^{ ( r ) } } ( 1 + e^{ - A \frac{ U }{ \tan \theta^{ ( r ) } } } ),
\end{align*}
we have
\begin{align*}
( \ref{PofThm4.1-6} ) & \ll \left( U^{ \Re ( w ) - 1 } + \left( \frac{ U }{ \tan \theta^{ ( r ) } } \right)^{ \Re ( w ) - 1 } \right) \left( \frac{ U }{ \tan \theta^{ ( r ) } } \right) \\
& \times \left( e^{ \left( \Im ( z ) + \frac{ r }{ 2 } \right) U } + e^{ \left( \Re ( z ) + \Im ( z ) \tan \theta^{ ( r ) } + \frac{ r }{ 2 } \tan \theta^{ ( r ) } - \varepsilon r \right) \frac{ U }{ \tan \theta^{ ( r ) } } } \right. \\
& \left. \phantom{ \times ( } + e^{ \left( \Im ( z ) + \taurz r \right) U } + e^{ \left( \Re ( z ) + \Im ( z ) \tan \theta^{ ( r ) } + \taurz r \tan \theta^{ ( r ) } \right) \frac{ U }{ \tan \theta^{ ( r ) } } } \right) \\
& \to 0 \quad ( M \to \infty ),
\end{align*}
where in the last limit we use the fact that
\begin{align}
\Im ( z ) + \taurz r < \Im ( z ) + \frac{ r }{ 2 } < 0 \label{PofThm4.1-7}
\end{align}
and
\begin{align*}
\begin{cases}
\dsp{ \Re ( z ) + \Im ( z ) \tan \theta^{ ( r ) } + \frac{ r }{ 2 } \tan \theta^{ ( r ) } - \varepsilon^{ ( r ) } r < 0 }, \\
\dsp{ \Re ( z ) + \Im ( z ) \tan \theta^{ ( r ) } + r \taurz \tan \theta^{ ( r ) } < 0 }
\end{cases}
\end{align*}
because $ ( w , z ) \in D_{ \theta^{ ( r ) } , \taurz } $ and $ \ \tan \theta^{ ( r ) } < \varepsilon^{ ( r ) } $. About the second term of $ ( \ref{PofThm4.1-5} ) $, by using Corollary \ref{Cor3.4} (iii) we have
\begin{align*}
\int_{ - 1 }^1 & \ll_r \int_{ - 1 }^1 e^{ - \Re ( z ) \sigma + \Im ( z ) U } ( U^2 e^{ ( \varepsilon^{ ( r ) } + \frac{ 1 }{ 2 } ) U } )^r U^{ \Re ( w ) - 1 } d \sigma \\
& \ll_z U^{ \Re ( w ) + 2 r - 1 } e^{ ( \Im ( z ) + ( \frac{ 1 }{ 2 } + \varepsilon^{ ( r ) } ) r ) U } \\
& \to 0 \quad ( M \to \infty ),
\end{align*}
where in the last limit we use $ \Im ( z ) + ( \frac{ 1 }{ 2 } + \varepsilon^{ ( r ) } ) r < 0 $. About the third term of $ ( \ref{PofThm4.1-5} ) $, by Corollary \ref{Cor3.4} (iii) we have
\begin{align}
\int_1^\infty & \ll_r \int_1^\infty e^{ - \Re ( z ) \sigma + \Im ( z ) U } e^{ \frac{ U }{ 2 } r } \max \{ | \sigma | , U \}^{ \Re ( w ) - 1 } d \sigma \notag \\
& = e^{ ( \Im ( z ) + \frac{ r }{ 2 } ) U } \left( \int_1^U e^{ - \Re ( z ) \sigma } U^{ \Re ( w ) - 1 } d \sigma + \int_U^\infty e^{ - \Re ( z ) \sigma } \sigma ^{ \Re ( w ) - 1 } d \sigma \right) \label{PofThm4.1-8} \\
& \ll e^{ ( \Im ( z ) + \frac{ r }{ 2 } ) U } ( U^{ \Re ( w ) - 1 } + 1 ) \label{PofThm4.1-9} \\
& \to 0 \quad ( M \to \infty ), \notag
\end{align}
where in transforming $ ( \ref{PofThm4.1-8} ) $ into $ ( \ref{PofThm4.1-9} ) $ we use $ \Re ( z ) > - ( \Im ( z ) + \frac{ r }{ 2 } ) \tan \theta^{ ( r ) } > 0 $ because $ ( w , z ) \in D_{ \theta^{ ( r ) } , \taurz } $, and in the last limit we use $ ( \ref{PofThm4.1-7} ) $. Hence, we obtain
\begin{align*}
F_{ \theta^{ ( r ) } } ( w , z , \{ \chi_j \}_{ j = 1 }^r ; \lambda ) = R_{ \theta^{ ( r ) } } ( w , z , \{ \chi_j \}_{ j = 1 }^r ).
\end{align*}

This completes the proof.
\end{proof}
In the following sections, it is necessary that the left-hand side of $ (\ref{Thm4.1-1}) $ be a meromorphic function of $ w $ at $ w = 0 $. To obtain the property we show a lemma. It is the generalization of the lemma proved by Hirano, Kurokawa and Wakayama{\upshape ~\cite[Lemma 1]{HKW}}.

Let $ \psi \in ( - \pi , \pi ] $ be any fixed real number and $ f(t) $ be a locally integrable function on $ \{ r e^{ i \psi } | r \in ( 0 , \infty ) \, \} $. We define
\begin{align*}
M_{ \psi } [ f : w] \coloneqq \int_0^{ \infty e^{ i \psi } } f ( t ) t^{ w - 1 } dt .
\end{align*}
Now, assume that $ f ( t ) $ satisfies
\begin{align*}
f( t ) = \begin{cases}
O ( t^{ - a + \varepsilon } ) & ( t \to 0 ), \\
O( t^{ - b - \varepsilon } ) & ( t \to \infty e^{ i \psi } )
\end{cases}
\end{align*}
for $ a , b \in \mathbb{ R } $ with $ a < b $ and $ M_\psi [ f : w ] $ converges absolutely, so is an analytic function, in $ a < \Re ( w ) < b $. Then, the following lemma holds.
\begin{lem} \label{Lem4.2}
Suppose that $ f ( t ) $ has the following approximate behaviors as $ t \to 0 $ and $ t \to \infty e^{ t \psi } $ : 
\begin{align}
f ( t ) \sim \begin{cases}
\dsp{ \sum_{ k = 0 }^\infty \sum_{ n = 0 }^{ N_1 ( k ) } A_1 ( n , k ) ( \log t )^n t^{ a_1 ( k ) } } & ( t \to 0 ), \\[+15pt]
\dsp{ \sum_{ k = 0 }^\infty \sum_{ n = 0 }^{ N_2 ( k ) } A_2 ( n , k ) ( \log t )^n t^{ a_2 ( k ) } } & ( t \to \infty e^{ i \psi } ),
\end{cases} \label{Lem4.2-1}%
\end{align}
where $ N_i ( k ) $ are non-negative and finite integers for each $ k $ and $ a_1 ( k ) $ and $ a_2 ( k ) $ are complex sequences with $ \Re ( a_1 ( k ) ) $ and $ \Re ( a_2 ( k ) ) $ monotonically increasing. Then $ M_\psi [ f : w ] $ has a meromorphic continuation into $ w \in \mathbb{ C } $ with poles at $ w = - a_1 ( k ) $ and $ w = - a_2 ( k ) $ for each $ k $. Especially the poles at $ s = - a_i ( k ) $ are simple if $ N_i ( k ) = 0 $. 
\end{lem}
\begin{proof}[\underline{\textbf{Proof of Lemma \ref{Lem4.2}}}]
First we define $ f_m ( t ) $ as
\begin{align*}
f_m ( t ) \coloneqq f ( t ) - \sum_{ k = 0 }^m \sum_{ n = 0 }^{ N_1 ( k ) } A_1 ( n , k ) ( \log t )^n t^{ a_1 ( k ) }.
\end{align*}
Then, in $ a < \Re ( w ) < b $, we have
\begin{align}
\begin{split}
M_\psi [ f : w ] &= \int_0^{ e^{ i \psi } } f_m ( t ) t^{ w - 1 } dt + \int_0^{ e^{ i \psi } } \sum_{ k = 0 }^m \sum_{ n = 0 }^{ N_1 ( k ) } A_1 ( n , k ) ( \log t )^n t^{ a_1 ( k ) + w - 1 } dt \\
& \quad + \int_{ e^{ i \psi } }^{ \infty e^{ i \psi } } f ( t ) t^{ w - 1 } dt.
\end{split} \label{PofLem4.2-1}
\end{align}
The first and third terms of the right-hand side of $ ( \ref{PofLem4.2-1} ) $ are analytic function of $ w $ in $ - \Re ( a_1 ( m + 1 ) ) < \Re ( w ) $ and in $ \Re ( w ) < b $ respectively. The second term becomes
\begin{align*}
\sum_{ k = 0 }^m \sum_{ n = 0 }^{ N_1 ( k ) } A_1 ( n , k ) \int_0^{ e^{ i \psi } } ( \log t )^n t^{ a_1 ( k ) + w - 1 } dt,
\end{align*}
and then by partial integration we can transform it into 
\begin{align*}
\sum_{ k = 0 }^m \sum_{ n = 0 }^{ N_1 ( k ) } \sum_{ r = 0 }^n A_1 ( n , k ) \frac{ ( - 1 )^r n ( n - 1 ) \cdots ( n - r + 1 ) ( i \psi )^{ n - r } e^{ i \psi ( w + a_1 ( k ) ) } }{ ( w + a_1 ( k ) )^{ r + 1 } }.
\end{align*}
Hence, we see that $ M_\psi [ f : w ] $ is a meromorphic function of $ w $ with having  poles at $ w = -a_1 ( k ) $ in $ - \Re ( a_1 ( m + 1 ) ) < \Re ( w ) < b $, especially the orders of which at $ s = - a_1 ( k ) $ are simple if $ N_1 ( k ) = 0 $. Since $ \Re ( a_1 ( m + 1 ) ) \to \infty \: ( m \to \infty ) $, it is shown that the meromorphy of $ M_\psi [ f : w ] $ in the left half plane $ \Re ( w ) < b $.

In a similar way, we can obtain a meromorphic continuation into the right half plane $ b \le \Re ( w ) $.
\end{proof}
From Lemma 4.2 the meromophy of the left-hand side of $ ( \ref{Thm4.1-1} ) $ follows.
\begin{cor} \label{Cor4.3}
If $ \left( \frac{ 1 }{ 2 } + \taurz \right) r \tan \theta^{ ( r ) } < \Re ( s ) \tan \theta^{ ( r ) } - \Im ( s ) < r \tan \theta^{ ( r ) } $ and $ \Re ( s ) > r ( 1 + \varepsilon^{ ( r ) } ) $, then $ L_{ \theta^{ ( r ) } }^{ ( 1 ) } \left( w , - i \left( s - \frac{ r }{ 2 } \right) , \{ \chi_j \}_{ j = 1 }^r \right) $ and \\
$ L_{ \theta^{ ( r ) } }^{ ( 2 ) } \left( w , - i \left( s - \frac{ r }{ 2 } \right) , \{ \chi_j \}_{ j = 1 }^r \right) $ are meromorphic functions of $ w $ on the whole $ \mathbb{ C } $.
\end{cor}
\begin{proof}[\underline{\textbf{Proof of Corollary \ref{Cor4.3}}}]
By the consideration about $ F_{ \theta^{ ( r ) } } ( w , z , \{ \chi_j \}_{ j = 1 }^r ; \lambda ) $ in the proof of Theorem \ref{Thm4.1}, $ L_{ \theta^{ ( r ) } }^{ ( 1 ) } ( w , z , \{ \chi_j \}_{ j = 1 }^r ) $ and $ L_{ \theta^{ ( r ) } }^{ ( 2 ) } ( w , z , \{ \chi_j \}_{ j = 1 }^r ) $ are holomorphic functions of $ w $ under the assumption that
\begin{align*}
( w , z ) \in D_{ \theta^{ ( r ) } , \taurz }, \Im ( z ) < - \left( \frac{ 1 }{ 2 } + \varepsilon^{ ( r ) } \right) r \mbox{ and } \Re ( w ) > r.
\end{align*}
We can remove $ \Re ( w ) > r $ because it follows from Theorem \ref{Thm3.2} (iii) that
\begin{align*}
e^{ - z t } \prod_{ j = 1 }^r l_{ \chi_j } ( t )
\end{align*}
and
\begin{align*}
& e^{ zt } \prod_{ j = 1 }^r \left( l_{ \chi_j } ( t ) + \sum_{ n = 1 }^\infty e^{ - \left( 2 n - 1 - \frac{ \chi_j ( -1 ) }{ 2 } \right) i t } \right. \\
& \left. \phantom{ \sum_{ n = 1 }^\infty \qquad } + \mutaujz e^{ i \taucjz t } + \mutaujz e^{ - i \taucjz t } + \mujz \right)
\end{align*}
which appear in $ L_{ \theta^{ ( r ) } }^{ ( i ) } ( w , z , \{ \chi_j \}_{ j = 1 }^r ) \, ( i = 1 , 2 ) $ satisfy the condition concerning $ t \to 0 $ in (\ref{Lem4.2-1}). By putting $ z = - i \left( s - \frac{ r }{ 2 } \right) $ we obtain the desired results.
\end{proof}
\section{\textbf{The zeta regularized product expression of $ \bm{ L ( s , \chi_1 ) } $}} \label{Sec5}
Our goal in this section is to prove Theorem \ref{Thm1.1}. We obtain an equation which links the ``factors series'' of $ L ( s , \chi_1 ) $ to prime numbers by calculating the both sides of $ ( \ref{Thm4.1-1} ) $ with $ r = 1 $ and then prove Theorem \ref{Thm1.1} .
\subsection{The ``key equation'' for $ r = 1 $}
\begin{lem} \label{Lem5.1}
Let $ ( w , z ) \in D_{ \theta^{ ( 1 ) } , \tauoz } $ satisfy $ \Im ( z ) < - ( \frac{ 1 }{ 2 } + \varepsilon^{ ( 1 ) } ) $ and $ \Re ( w ) > 1 $. Then,
\begin{align*}
& L_{ \theta^{ ( 1 ) } }^{ ( 1 ) } ( w , z , \chi_1 ) = \sum_{ \Re ( \tauco ) > 0 } \frac{ 1 }{ ( z + \tauco )^w }, \\
&L_{ \theta^{ ( 1 ) } }^{ ( 2 ) } ( w , z , \chi_1 ) \\
&= e^{ \pi i w } \left( \sum_{ \Re ( \taucobar ) > 0 } \frac{ 1 }{ ( \taucobar - z )^w } + \sum_{ n = 1 }^\infty \frac{ 1 }{ \left( - z + \left( 2 n - 1 - \frac{ \chi_1 ( -1 ) }{ 2 } \right) i \right)^w } \right. \\
&\left. \qquad \qquad + \frac{ \mutauoz }{ \left( - z - i \taucoz \right)^w } + \frac{ \mutauoz }{ \left( - z + i \taucoz \right)^w } + \frac{ \muoz }{ ( - z  )^w } \right).
\end{align*}
\end{lem}
\begin{proof}[\underline{\textbf{Proof of Lemma \ref{Lem5.1}}}]
Since $ ( w , z ) \in D_{ \theta^{ ( 1 ) } , \tauoz } $ and $ \Re ( \tauco ) > \varepsilon^{ ( 1 ) } > \tan \theta^{ ( 1 ) } $, we have
\begin{align*}
& \Re ( z + \tauco ) + \Im ( z + \tauco) \tan \theta^{ ( 1 ) } > \varepsilon^{ ( 1 ) } - \tan \theta^{ ( 1 ) } > 0,
\end{align*}
and from this we find $ \arg ( z + \tauco ) \in \left( \theta^{ ( 1 ) } - \frac{ \pi }{ 2 } , \theta^{ ( 1 ) } + \frac{ \pi }{ 2 } \right) $. Therefore, by using Lemma \ref{Lem2.8} as $ \psi = - \theta^{ ( 1 ) } $ we obtain
\begin{align*}
L_{ \theta^{ ( 1 ) } }^{ ( 1 )} ( w , z , \chi_1 ) & = \frac{ 1 }{ \Gamma ( w ) } \sum_{ \Re ( \tauco ) > 0 } \int_0^{ \infty e^{ - i \theta^{ ( 1 ) } }} e^{ - ( z + \tauco ) t } t^{ w - 1 } dt \\
& = \sum_{ \Re ( \tauco ) > 0 } \frac{ 1 }{ ( z + \tauco )^w }.
\end{align*}

In a similar way as $ L_{ \theta^{ ( 1 ) } }^{ ( 1 ) } ( w , z , \chi_1 ) $ we can reach the desired result concerning $ L_{ \theta^{ ( 1 ) } }^{ ( 2 ) } ( w , z , \chi_1 ) $.
\end{proof}
\begin{lem} \label{Lem5.2}
If $ ( \frac{ 1 }{ 2 } + \tauoz ) \tan \theta^{ ( 1 ) } < \Re ( s ) \tan \theta^{ ( 1 ) } - \Im ( s ) < \tan \theta^{ ( 1 ) }, \, \Re ( s ) > 1 + \varepsilon^{ ( 1 ) } $ and $ \Re ( w ) > 1 $ then we have
\begin{align}
& L_{ \theta^{ ( 1 ) } }^{ ( 1 ) } \left( w , - i \left( s - \frac{ 1 }{ 2 } \right) , \chi_1 \right) = e^{ \frac{ \pi i w }{ 2 } } \sum_{ \Im ( \rhocobar ) < 0 } \frac{ 1 }{ ( s - \rhocobar )^w }, \label{Lem5.2-1} \\
& L_{ \theta^{ ( 1 ) } }^{ ( 2 ) } \left( w , - i \left( s - \frac{ 1 }{ 2 } \right) , \chi_1 \right) \notag \\
\begin{split}
& = e^{ \frac{ \pi i w }{ 2 } } \left( \sum_{ \Im ( \rhocobar ) > 0 } \frac{ 1 }{ ( s - \rhocobar )^w } + \sum_{ n = 1 }^\infty \frac{ 1 }{ \left( s + 2n - \frac{ 3 + \chi_1 ( -1 ) }{ 2 } \right)^w } \right. \\
& \qquad \qquad \left. + \frac{ \mutauoz }{ \left( s - \frac{ 1 }{ 2 } - \taucoz \right)^w } + \frac{ \mutauoz }{ \left( s - \frac{ 1 }{ 2 } + \taucoz \right)^w } + \frac{ \muoz }{ \left( s - \frac{ 1 }{ 2 } \right)^w } \right),
\end{split} \label{Lem5.2-2}
\end{align}
where the argument lies in $ \left( - \frac{ \pi }{ 2 } , \frac{ \pi }{ 2 } \right) $. The serieses in (\ref{Lem5.2-1}) and (\ref{Lem5.2-2}) converge absolutely, locally and uniformly in the given $ ( w , s ) $-region above.
\end{lem}
\begin{proof}[\underline{\textbf{Proof of Lemma \ref{Lem5.2}}}]
Putting $ z = - i ( s - \frac{ 1 }{ 2 } ) $ in Lemma \ref{Lem5.1}, we obtain the conditions concerning $ ( w , s ) $ and have
\begin{align*}
L_{ \theta^{ ( 1 ) } }^{ ( 1 ) } \left( w , - i \left( s - \frac{ 1 }{ 2 } \right) , \chi_1 \right) & = \sum_{ \Re ( \tauco ) > 0 } \frac{ 1 }{ ( - i ( s - \frac{ 1 }{ 2 } ) + \tauco )^w } \\
& : \arg \left( - i \left( s - \frac{ 1 }{ 2 } \right) + \tauco \right) \in \left( \theta^{ ( 1 ) } - \frac{ \pi }{ 2 } , \theta^{ ( 1 ) } + \frac{ \pi }{ 2 } \right) \\
& = \sum_{ \Re ( \tauco ) > 0 } \frac{ e^{ \frac{ \pi i w }{ 2 } } }{ ( s - \frac{ 1 }{ 2 } + i \tauco )^w } \\
& : \arg \left( s - \frac{ 1 }{ 2 } + i \tauco \right) \in ( \theta^{ ( 1 ) } , \theta^{ ( 1 ) } + \pi ) \\
& = e^{ \frac{ \pi i w }{ 2 } } \sum_{ \Im ( \rhocobar ) < 0 } \frac{ 1 }{ ( s - \rhocobar )^w } \\
& : \arg ( s - \rhocobar ) \in ( \theta^{ ( 1 ) } , \theta^{ ( 1 ) } + \pi ).
\end{align*}
Now, since $ \Re ( s - \rhocobar ) > \varepsilon^{ ( 1 ) } > 0 $ is derived from $ \Re ( s ) > 1 + \varepsilon^{ ( 1 ) } $ and $ 0 < \Re ( \rhocobar ) < 1 $, we find
\begin{align*}
\arg ( s - \rhocobar ) \in \left( \theta^{ ( 1 ) } , \frac{ \pi }{ 2 } \right) \subset \left( - \frac{ \pi }{ 2 } , \frac{ \pi }{ 2 } \right).
\end{align*}

In the same way, we obtain $ ( \ref{Lem5.2-2} ) $.

The absolute and locally uniform convergences of the serieses in (\ref{Lem5.2-1}) and (\ref{Lem5.2-2}) in $ \Re ( s ) > 1 $ and $ \Re ( w ) > 1 $ are easily derived from
\begin{align*}
\sharp \{ \rhoco \, | \,  \Im( \rhoco ) \in ( T , T + 1 ] \} = O ( \log T ).
\end{align*}
The desired convergency follows immediately from $ \Re ( s ) > 1 $ and $ \Re ( w ) > 1 $ including the given $ ( w , s ) $-region. 
\end{proof}
\begin{lem} \label{Lem5.3}
If $ ( \frac{ 1 }{ 2 } + \tauoz ) \tan \theta^{ ( 1 ) } < \Re ( s ) \tan \theta^{ ( 1 ) } - \Im ( s ) < \tan \theta^{ ( 1 ) } $, $ \Re ( s ) > 1 + \varepsilon^{ ( 1 ) } $ and $ \ \Re ( w ) > 1 $ then we have
\begin{align}
R_{ \theta^{ ( 1 ) } } \left( w , - i \left( s - \frac{ 1 }{ 2 } \right) , \chi_1 \right) = - \frac{ e^{ \frac{ \pi i w }{ 2 } } }{ \Gamma ( w ) } \sum_{ p , m } \bar{ \chi_1 } ( p^m ) p^{ - m s } ( m \log p )^{ w - 1 } \log p. \label{Lem5.3-1}
\end{align}
The series converges absolutely and uniformly on any compact subset of $ \{ ( w , s ) \in \mathbb{ C }^2 \ | \ \Re ( s ) > 1 \} $.
\end{lem}
\begin{proof}[\underline{\textbf{Proof of Lemma \ref{Lem5.3}}}]
By Theorem \ref{Thm3.2} (ii) and (iv), we find that the residue in $ R_{ \theta^{ ( 1 ) } } \left( w , - i \left( s - \frac{ 1 }{ 2 } \right) , \chi_1 \right) $ is equal to
\begin{align*}
\underset{ t = i m \log p }{ \mathrm{ Res } } e^{ i ( s - \frac{ 1 }{ 2 } ) t } l_{ \chi_1 } ( t ) t^{ w - 1 } & = \underset{ t = i m \log p }{ \mathrm{ Res } } e^{ i ( s - \frac{ 1 }{ 2 } ) t } \left( \frac{ i t }{ 2 \pi } e^{ - \frac{ i t }{ 2 } } \frac{ \bar{ \chi_1 } ( p^m ) p^{ - m } }{ m ( t - i m \log p ) } \right) t^{ w - 1 } \\
& = - \frac{ i^{ w - 1 } }{ 2 \pi } \bar{ \chi_1 } ( p^m ) p^{ - m s } ( m \log p )^{ w - 1 } \log p.
\end{align*}
From this $ ( \ref{Lem5.3-1} ) $ follows.

From Lemma \ref{Lem2.7} (iii) it follows that the series in $ ( \ref{Lem5.3-1} ) $ converges absolutely and uniformly on any compact subset of $ \{ ( w , s ) \in \mathbb{ C }^2 \, | \, \Re ( s ) > 1 \} $.
\end{proof}
By using the above three lemmas we derive the desired equation.
\begin{thm} \label{Thm5.4}
If $ ( \frac{ 1 }{ 2 } + \tauoz ) \tan \theta^{ ( 1 ) } < \Re ( s ) \tan \theta^{ ( 1 ) } - \Im ( s ) < \tan \theta^{ ( 1 ) }, \, \Re ( s ) > 1 + \varepsilon^{ ( 1 ) } $ and $ \Re ( w ) > 1 $, we have
\begin{align}
\left. \begin{array}{@{}l}
\dsp{ \sum_{ \Im ( \rhoco ) \not = 0 } \frac{ 1 }{ ( s - \rhoco )^w } + \sum_{ n = 1 }^\infty \frac{ 1 }{ \left( s + 2n - \frac{ 3 + \chi_1 ( -1 ) }{ 2 } \right)^w } } \\[+15pt]
+ \dsp{ \frac{ \mutauoz }{ \left( s - \frac{ 1 }{ 2 } - \taucoz \right)^w } + \frac{ \mutauoz }{ \left( s - \frac{ 1 }{ 2 } + \taucoz \right)^w } + \frac{ \muoz }{ \left( s - \frac{ 1 }{ 2 } \right)^w } } \\[+20pt]
= \dsp{ - \frac{ 1 }{ \Gamma ( w ) } \sum_{ p , m } \chi_1 ( p^m ) p^{ - m s } ( m \log p )^{ w - 1 } \log p }.
\end{array} \right\} \label{Thm5.4-1}
\end{align}
\end{thm}
\begin{proof}[\underline{\textbf{Proof of Theorem \ref{Thm5.4}}}]
We put $ r = 1 $ and $ z = - i \left( s - \frac{ 1 }{ 2 } \right) $ in Theorem \ref{Thm4.1} and then by applying Lemma \ref{Lem5.2} and \ref{Lem5.3} we have
\begin{align}
\left. \begin{array}{@{}l}
\dsp{ \sum_{ \Im ( \rhocobar ) \not = 0 } \frac{ 1 }{ ( s - \rhocobar )^w } + \sum_{ n = 1 }^\infty \frac{ 1 }{ \left( s + 2n - \frac{ 3 + \chi_1 ( -1 ) }{ 2 } \right)^w } } \\[+15pt]
+ \dsp{ \frac{ \mutauoz }{ \left( s - \frac{ 1 }{ 2 } - \taucoz \right)^w } + \frac{ \mutauoz }{ \left( s - \frac{ 1 }{ 2 } + \taucoz \right)^w } + \frac{ \muoz }{ \left( s - \frac{ 1 }{ 2 } \right)^w } } \\[+20pt]
= \dsp{ - \frac{ 1 }{ \Gamma ( w ) } \sum_{ p , m } \bar{ \chi_1 } ( p^m ) p^{ - m s } ( m \log p )^{ w - 1 } \log p },
\end{array} \right\} \label{PofThm5.4-1}
\end{align}
under the conditions that
\begin{align*}
& \left( \tauoz + \frac{ 1 }{ 2 } \right) \tan \theta^{ ( 1 ) } < \Re ( s ) \tan \theta^{ ( 1 ) } - \Im ( s ) < \tan \theta^{ ( 1 ) }, \\
 & \Re ( s ) > 1 + \varepsilon^{ ( 1 ) } \mbox{ and } \Re ( w ) > 1.
\end{align*}
Then, replacing $ \bar{ \chi_1 } $ with $ \chi_1 $ in $ ( \ref{PofThm5.4-1} ) $, we obtain $ ( \ref{Thm5.4-1} ) $.
\end{proof}
\subsection{Proof of Theorem 1.1}
\begin{proof}
The left-hand side of $ ( \ref{Thm5.4-1} ) $ is a meromorphic function of $ w $ on the whole $ \mathbb{ C } $ by Corollary \ref{Cor4.3}. Hence, by using the definition of the zeta regularized product we have
\begin{align}
& \exp \left( - \underset{ w = 0 }{ \mathrm{ Res } } \left( \frac{ \mbox{the left-hand side of } ( \ref{Thm5.4-1} ) }{ w^2 } \right) \right) \notag \\
\begin{split}
& = \Rprod_{ \Im ( \rhoco ) \not = 0 } ( s - \rhoco ) \Rprod_{ n = 1 }^\infty \left( s + 2 n - \frac{ 3 + \chi_1 ( - 1 ) }{ 2 } \right) \\
& \quad \times \left( s - \frac{ 1 }{ 2 } - \taucoz \right)^{ \mutauoz } \left( s - \frac{ 1 }{ 2 } + \taucoz \right)^{ \mutauoz } \left( s - \frac{ 1 }{ 2 } \right)^{ \muoz }.
\end{split} \label{PofThm1.1-1}
\end{align}

On the other hand, since $ \frac{ 1 }{ \Gamma ( w ) } = w + O ( w^2 ) \ ( w \to 0 ) $, we have
\begin{align*}
\exp \left( - \underset{ w = 0 }{ \mathrm{ Res } } \left( \frac{ \mbox{the right-hand side of } ( \ref{Thm5.4-1} ) }{ w^2 } \right) \right) & = \exp \left( \sum_{ p , m } \frac{ \chi_1 ( p^m ) p^{ - m s } }{ m } \right) \\
& = \prod_{ p } ( 1 - \chi_1 ( p ) p^{ - s } )^{ - 1 }.
\end{align*}
By the property of the zeta regularized products, $ ( \ref{PofThm1.1-1} ) $ is a meromorphic function on the whole $ \mathbb{ C } $. Hence $ ( \ref{Thm1.1-1} ) $ holds.
\end{proof}
\section{\textbf{The Euler product expression of $ \bm{ ( L_{ \chi_1 } \otimes_{ \mathbb{ F }_1 } L_{ \chi_2 } ) ( s ) } $}} \label{Sec6}
In a similar way as section 5, we show Theorem \ref{Thm1.3}.
\subsection{The ``key equation'' for $ r = 2 $}
\begin{lem} \label{Lem6.1}
If $ ( 2 \tautz + 1 ) \tan \theta^{ ( 2 ) } < \Re ( s ) \tan \theta^{ ( 2 ) } - \Im ( s ) < 2 \tan \theta^{ ( 2 ) } $, $ \Re ( s ) > 2 ( 1 + \varepsilon^{ ( 2 ) } ) $ and $ \ \Re ( w ) > 2 $ then we have
\begin{align*}
& L_{ \theta^{ ( 2 ) } }^{ ( 1 ) } ( w , -i ( s - 1 ) , \{ \chi_j \}_{ j = 1 }^2 ) = e^{ \frac{ \pi i w }{ 2 } } \sum_{ \Im ( \rhocobar ) , \Im ( \rhoctbar ) < 0 } \frac{ 1 }{ ( s - \rhocobar - \rhoctbar )^w }, \\
& L_{ \theta^{ ( 2 ) } }^{ ( 2 ) } ( w , - i ( s - 1 ) , \{ \chi_j \}_{ j = 1 }^2 ) \\ 
& = - e^{ \frac{ \pi i w }{ 2 } } \left( \sum_{ \Im ( \rhocobar ) , \Im ( \rhoctbar ) > 0 } \frac{ 1 }{ ( s - \rhocobar - \rhoctbar )^w } \right. \\
& \quad + \sum_{ ( a , b ) \in \{ ( 1 , 2 ) , ( 2 , 1 ) \} } \left( \sum_{ \Im ( \rhocabar ) > 0 } \sum_{ n = 1 }^\infty \frac{ 1 }{ \left( s - \rhocabar + 2n - \frac{ 3 + \chi_b ( -1 ) }{ 2 } \right)^w } \right. \\
& \quad + \sum_{ \Im ( \rhocbbar ) > 0 } \frac{ \mutauaz }{ \left( s - \rhocbbar - \frac{ 1 }{ 2 } - \taucaz \right)^w } + \sum_{ \Im ( \rhocbbar ) > 0 } \frac{ \mutauaz }{ \left( s - \rhocbbar - \frac{ 1 }{ 2 } + \taucaz \right)^w } \\
& \quad + \sum_{ \Im ( \rhocbbar ) > 0 } \frac{ \muaz }{  \left( s - \rhocbbar - \frac{ 1 }{ 2 } \right)^w } + \sum_{ n = 1 }^\infty \frac{ \mutauaz }{ \left( s + 2n - 2 - \frac{ \chi_b ( -1 ) }{ 2 } - \taucaz \right)^w } \\
& \quad + \left. \sum_{ n = 1 }^\infty \frac{ \mutauaz }{ \left( s + 2n - 2 - \frac{ \chi_b ( -1 ) }{ 2 } + \taucaz \right)^w } + \sum_{ n = 1 }^\infty \frac{ \muaz }{ \left( s + 2n - 2 - \frac{ \chi_b ( -1 ) }{ 2 } \right)^w } \right) \\
& \quad + \sum_{ n_1 = 1 }^\infty \sum_{ n_2 = 1 }^\infty \frac{ 1 }{ \left( s + 2 n_1 + 2 n_2 - 3 - \frac{ \chi_1 ( -1 ) + \chi_2 ( -1 ) }{ 2 } \right)^w } \\
& \quad + \frac{ \mutauoz \mutautz }{ \left( s - 1 - \taucoz - \tauctz \right)^w } + \frac{ \mutauoz \mutautz }{ \left( s - 1 - \taucoz + \tauctz \right)^w } \\
& \quad + \frac{ \mutauoz \mutautz }{ \left( s - 1 + \taucoz - \tauctz \right)^w } + \frac{ \mutauoz \mutautz }{ \left( s - 1 + \taucoz + \tauctz \right)^w } \\
& \quad \left. + \sum_{ ( a , b ) \in \{ ( 1 , 2 ) , ( 2 , 1 ) \} } \left( \frac{ \mutauaz \mubz }{ \left( s - 1 - \taucaz \right)^w } + \frac{ \mutauaz \mubz }{ \left( s - 1 + \taucaz \right)^w } \right) + \frac{ \muoz \mutz }{ ( s - 1 )^w } \right).
\end{align*}
The serieses which appear here converge absolutely, locally and uniformly in the given $ ( w , s ) $-region above. 
\end{lem}
\begin{proof}[\underline{\textbf{Proof of Lemma \ref{Lem6.1}}}]
In a similar way as Lemma \ref{Lem5.1} and \ref{Lem5.2} we can prove these.
\end{proof}
\begin{lem} \label{Lem6.2}
If $ ( 2 \tautz + 1 ) \tan \theta^{ ( 2 ) } < \Re ( s ) \tan \theta^{ ( 2 ) } - \Im ( s ) < 2 \tan \theta^{ ( 2 ) } $, $ \Re ( s ) > 2 ( 1 + \varepsilon^{ ( 2 ) } ) $ and $ \Re ( w ) > 2 $ then we have
\begin{align}
R_{ \theta^{ ( 2 ) } } ( w , - i ( s - 1 ) , \{ \chi_j \}_{ j = 1 }^2 ) = \frac{ e^{ \frac{ \pi i w }{ 2 } } }{ \Gamma ( w ) } \sum_{ k = 1 }^{ 10 } E_k ( w , s , \{ \bar{ \chi_j } \}_{ j = 1 }^2 ). \label{Lem6.2-1}
\end{align}
\end{lem}
\begin{proof}[\underline{\textbf{Proof of Lemma \ref{Lem6.2}}}]
Let $ p $ and $ m $ be any fixed prime number and positive integer respectively. By Theorem \ref{Thm3.2} (ii) we have
\begin{align*}
& l_{ \chi_1 } ( t ) \cdot l_{ \chi_2 } ( t ) \\
& = \bar{ \chi_1 } ( p^m ) \bar{ \chi_2 } ( p^m ) \left( \frac{ i t e^{ - \frac{ i t }{ 2 } } p^{ - m } }{ 2 \pi m ( t - i m \log p ) } \right)^2 \\
& \quad + \sum_{ ( a , b ) \in \{ ( 1 , 2 ) , ( 2 , 1 ) \} } \frac{ i t e^{ - \frac{ i t }{ 2 } } \bar{ \chi_a } ( p^m ) p^{ - m } }{ 2 \pi m ( t - i m \log p ) } \Bigg( - \frac{ i t }{ 2 \pi } e^{ - \frac{ i t }{ 2 } } \sum_{ \scriptsize{ \begin{array}{c}
q , n \\
p^m \not = q^n
\end{array} } } \frac{ \bar{ \chi_b } ( q^n ) q^{ - n } }{ n ( t - i n \log q ) } \\
& \quad - \frac{ e^{ i \left( \alpha + \frac{ 1 }{ 2 } \right) t } }{ 2 \pi } \left( it \sum_{ q , n } \frac{ \chi_b ( q^n ) q^{ - n ( 1 + \alpha ) } }{ n ( t + i n \log q ) } - i t \sum_{ n = 1 }^\infty \frac{ \chi_b ( - 1 ) e^{ - i \alpha m \pi } }{ n ( t - n \pi ) } \right. \\
& \quad + i \log \left( \frac{ \chi_b ( - 1 ) \Gamma ( 1 + \alpha ) N_b^\alpha G ( \bar{ \chi_b } ) }{ ( 2 \pi )^{ 1 + \alpha } } \right) - \frac{ ( 1 + \alpha ) \pi }{ 2 } \\
& \quad \left. + \frac{ 1 }{ t } \left( \gamma + \log \left( \frac{ 2 \pi }{ N_b } \right) + \frac{ \pi i }{ 2 } \right) - \frac{ 1 }{ t } \int_0^\infty \frac{ 1 }{ e^u - 1 } \cdot \frac{ u - i t ( 1 - e^{ - \alpha u } ) }{ u - i t } du \right) \\
& \quad - \frac{ t }{ 2 \pi } e^{ \frac{ i t }{ 2 } } \int_{ S^{ ( 2 ) } ( \pi \to 0 ) } e^{ - i s t } \log L ( s , \bar{ \chi_b } ) ds - \frac{ i e^{ - \frac{ \chi_b ( - 1 ) }{ 2 } it } }{ 2 \sin t } \\
& \quad - \mutaubz ( e^{ i \mutaubz t } + e^{ - i \mutaubz t } ) - \mubz \Bigg) \\
& \quad + ( \mbox{the holomorphic parts at } t = i m \log p ).
\end{align*}
Applying this to
\begin{align*}
R_{ \theta^{ ( 2 ) } } ( w , - i ( s - 1 ) , \{ \chi_j \}_{ j = 1 }^2 ) = \frac{ 2 \pi i }{ \Gamma ( w ) } \sum_{ p , m } \underset{ t = i m \log p }{ \mathrm{ Res } } ( e^{ i ( s - 1 ) t } l_{ \chi _1 } ( t ) \cdot l_{ \chi_2 } ( t ) t^{ w - 1 } )
\end{align*}
leads to $ ( \ref{Lem6.2-1} ) $.
\end{proof}
In the following lemma we show the convergencies of $ E_k ( w , s , \{ \chi_j \}_{ j = 1 }^2 ) $ which can be proved in almost the same way as Akatsuka's method used in \cite[Theorem 1.2]{Aka2}
\begin{lem} \label{Lem6.3}
For $ k \in \{ 1 , 2 , \cdots , 10 \} $, $ E_k ( w , s , \{ \chi_j \}_{ j = 1 }^2 ) $ converges absolutely and uniformly on any compact subset of $ \{ ( w , z ) \in \mathbb{ C }^2 \ | \ \Re ( s ) > \beta_k \} $, where
\begin{align*}
\beta_k = \begin{cases}
1 & ( k = 1 ), \\
2 & ( k = 2 , 6 ), \\
1 - \alpha & ( k = 3 , 4 , 7 ), \\[+10pt]
\max \left\{ \frac{ 1 - \chi_1 ( -1 ) }{ 2 } , \frac{ 1 - \chi_2 ( - 1 ) }{ 2 } \right\} & ( k = 5 ), \\[+10pt]
\frac{ 3 }{ 2 } + \tautz & ( k = 8 ), \\[+10pt]
\frac{ 3 }{ 2 } - \tautz & ( k = 9 ), \\[+10pt]
\frac{ 3 }{ 2 } & ( k = 10 ).
\end{cases}
\end{align*}
\end{lem}
\begin{proof}[\underline{\textbf{Proof of Lemma \ref{Lem6.3}}}]
The desired results follow from Lemma \ref{Lem2.7} (iii) immediately except for $ E_2 ( w , s , \{ \chi_j \}_{ j = 1 }^2 ) $, $ \ E_3 ( w , s , \{ \chi_j \}_{ j = 1 }^2 ) $ and $ \ E_4 ( w , s , \{ \chi_j \}_{ j = 1 }^2 ) $.

Concerning $ E_4 ( w , s , \{ \chi_j \}_{ j = 1 }^2 ) $, we can easily prove its absolute and locally uniform convergence by Lemma \ref{Lem2.7} (iii).

We consider $ E_3 ( w , s , \{ \chi_j \}_{ j = 1 }^2 ) $. Let $ ( w , s ) \in \mathbb{ C }^2 $ satisfy $ \Re ( s ) > 1 - \alpha + \delta $ and $ A \le \Re ( w ) \le B $ for any fixed real numbers $ \varepsilon $, $ A $ and $ B $ with $ \delta > 0 $ and $ A < B $. Then, for any prime numbers $ p , q $ and any $ m , n \in \mathbb{ Z }_{ \ge 1 } $ we have
\begin{align*}
& \left| \frac{ \chi_a ( p^m ) \bar{ \chi_b } ( q^n ) p^{ - m( s + \alpha ) } q^{ - n } ( m \log p )^w \log p }{ n ( m \log p + n \log q ) } \right| \\
& \le \begin{cases}
\dsp{ \frac{ 2^{ - ( 1 + \delta ) } q^{ - n } ( \log 2 )^{ A + 1 } }{ n^2 \log q } } & ( p = 2 , m = 1 ), \\[+15pt]
\dsp{ \frac{ p^{ - m ( 1 + \delta ) } q^{ - n } ( m \log p )^B \log p }{ n^2 \log q } } & ( \mbox{otherwise} ),
\end{cases}
\end{align*}
where $ ( a , b ) \in \{ ( 1 , 2 ) , ( 2 , 1 ) \} $. From Lemma \ref{Lem2.7} (ii) we have
\begin{align*}
\sum_{ q , n } \frac{ 2^{ - ( 1 + \delta ) } q^{ - n } ( \log 2 )^{ A + 1 } }{ n^2 \log q } < \infty.
\end{align*}
From Lemma \ref{Lem2.7} (ii), (iii) we have
\begin{align*}
& \sum_{ \scriptsize{ \begin{array}{c}
p , m , q , n \\
p^m \ge 3
\end{array} } } \frac{ p^{ - m ( 1 + \varepsilon ) } q^{ - n } ( m \log p )^B \log p }{ n^2 \log q } \\
& \le \left( \sum_{ p , m } p^{ - m ( 1 + \varepsilon ) } ( m \log p )^B \log p \right) \left( \sum_{ q , n } \frac{ q^{ - n } }{ n^2 \log q } \right) \\
& < \infty.
\end{align*}
Hence, we find that $ E_3 ( w , s , \{ \chi_j \}_{ j = 1 }^2 ) $ converges absolutely and uniformly on any compact subset of $ \{ ( w , s ) \in \mathbb{ C }^2 \ | \ \Re ( s ) >1 - \alpha \} $.

We consider $ E_2 ( w , s , \{ \chi_j \}_{ j = 1 }^2 ) $. Let $ ( w , s ) \in \mathbb{ C }^2 $ satisfy $ \Re ( s ) > 2 + \delta $ and $ A \le \Re ( w ) \le B $ for any fixed real numbers $ \delta $, $ A $ and $ B $ with $ \delta > 0 $ and $ A < B $. Then, for any prime numbers $ p , q $ and any $ m , n \in \mathbb{ Z }_{ \ge 1 } $ we have
\begin{align*}
& \left| \frac{ \chi_a ( p^m ) \chi_b ( q^n ) p^{ - m ( s - 1 ) } q^{ - n } ( m \log p )^w \log p }{ n ( m \log p - n \log q ) } \right| \\
& \le \begin{cases}
\dsp{ \frac{ 2^{ - ( 1 + \delta ) } q^{ - n } ( \log 2 )^{ A + 1 } }{ n ( n \log q - \log 2 ) } } & ( p = 2 , m = 1 ), \\[+15pt]
\dsp{ \frac{ p^{ - m ( 1 + \delta ) } q^{ - n } ( m \log p )^B \log p }{ n | m \log p - n \log q | } } & ( \mbox{otherwise} ),
\end{cases}
\end{align*}
where $ ( a , b ) \in \{ ( 1 , 2 ) , ( 2 , 1 ) \} $. In the case of $ ( p , m ) = ( 2 , 1 ) $, from $ \log x - \log 2 \ge \left( 1 - \frac{ \log 2 }{ \log 3 } \right) \log x $ for any $ x \in \mathbb{ R }_{ \ge 3 } $ and Lemma \ref{Lem2.7} (ii) it follows that
\begin{align*}
\sum_{ \scriptsize{ \begin{array}{c}
q , n \\
q^n \ge 3
\end{array} } } \frac{ 2^{ - ( 1 + \delta ) } q^{ - n } ( \log 2 )^{ A + 1 } }{ n ( n \log q - \log 2 ) } \le \frac{ 2^{ - ( 1 + \delta ) } ( \log 2 )^{ A + 1 } }{ 1 - \frac{ \log 2 }{ \log 3 } } \dsp{ \sum_{ q , n } \frac{ q^{ - n } }{ n^2 \log q } } < \infty.
\end{align*}
In the case of $ ( p , m ) \not = ( 2 , 1 ) $, we have
\begin{align}
& \sum_{ \scriptsize{ \begin{array}{c}
p , m , q , n \\
p^m \ge 3 \\
q^n \not = p^m
\end{array} } } \frac{ p^{ - m ( 1 + \delta ) } q^{ - n } ( m \log p )^B \log p }{ n | m \log p - n \log q | } \notag \\
& \le \sum_{ \scriptsize{ \begin{array}{c}
p , m , q , n \\
q^n \not = p^m
\end{array} } } \frac{ p^{ - m ( 1 + \delta ) } q^{ - n } ( m \log p )^B \log p }{ n | m \log p - n \log q | } \notag \\
& = \sum_{ \scriptsize{ \begin{array}{c}
p , m , q , n \\
q^n < p^m
\end{array} } } + \sum_{ \scriptsize{ \begin{array}{c}
p , m , q , n \\
p^m < q^n < p^{ 2 m }
\end{array} } } + \sum_{ \scriptsize{ \begin{array}{c}
p , m , q , n \\
q^n \ge p^{ 2 m }
\end{array} } }. \label{PofLem6.3-1}
\end{align}
Concerning the third term of $ ( \ref{PofLem6.3-1} ) $, we have $ n \log q - m \log p \ge \frac{ n \log q }{ 2 } $ because $ 2 m \log p \le n \log q $. Therefore, from Lemma \ref{Lem2.7} (ii), (iii) we have
\begin{align}
( \mbox{the third term of } ( \ref{PofLem6.3-1} ) ) \le 2 \left( \sum_{ p , m } p^{ - m ( 1 + \delta ) } ( m \log p )^B \log p \right) \left( \sum_{ q , n } \frac{ q^{ - n } }{ n^2 \log q } \right) < \infty. \label{PofLem6.3-2}
\end{align}
Concerning the second term of $ ( \ref{PofLem6.3-1} ) $, from Lemma \ref{Lem2.7} (i) we have
\begin{align*}
\sum_{ \scriptsize{ \begin{array}{c}
q , n \\
p^m < q^n <p^{ 2 m }
\end{array} } } \frac{ q^{ - n } }{ n | m \log p - n \log q | } & \le \sum_{ \scriptsize{ \begin{array}{c}
q , n \\
p^m < q^n < p^{ 2 m }
\end{array} } } q^{ - n } \frac{ q^n }{ q^n - p^m } \\
& \le \sum_{ l = 1 }^{ p^{ 2 m } - p^m - 1 } \frac{ 1 }{ ( p^m + l ) - p^m } \ll m \log p.
\end{align*}
Hence, from Lemma \ref{Lem2.7} (iii) we find
\begin{align}
( \mbox{the second term of }( \ref{PofLem6.3-1} ) ) \ll \sum_{ p , m } p^{ m ( 1 + \delta ) } ( m \log p )^{ B + 1 } \log p < \infty. \label{PofLem6.3-3}
\end{align}
Concerning the first term of $ ( \ref{PofLem6.3-1} ) $, from Lemma \ref{Lem2.7} (i) we have
\begin{align*}
& \sum_{ \scriptsize{ \begin{array}{c}
q , n \\
q^n < p^m
\end{array} } } \frac{ q^{ - n } }{ n | m \log p - n \log q | } \\
& \le \sum_{ \scriptsize{ \begin{array}{c}
q , n \\
q^n < p^m
\end{array} } } q^{ - n } \frac{ p^m }{ p^m - q^n } \le p^m \sum_{ l = 1 }^{ p^m - 1 } \frac{ 1 }{ l ( p^m - l ) } = \sum_{ l = 1 }^{ p^m - 1 } \frac{ 1 }{ l } + \sum_{ l  = 1 }^{ p^m - 1 } \frac{ 1 }{ p^m - l } \\
& \ll m \log p.
\end{align*}
Hence, from Lemma \ref{Lem2.7} (iii) we find
\begin{align}
( \mbox{the first term of } \ref{PofLem6.3-1} ) \ll \dsp{ \sum_{ p , m } p^{ - m ( 1 + \delta ) } ( m \log p )^{ B + 1 } \log p } < \infty. \label{PofLem6.3-4}
\end{align}
From $ ( \ref{PofLem6.3-2} ) $, $ ( \ref{PofLem6.3-3} ) $ and $ ( \ref{PofLem6.3-4} ) $ it follows that $ ( \ref{PofLem6.3-1} ) $ converges. This completes the proof.
\end{proof}
From Lemma \ref{Lem6.1}, Lemma \ref{Lem6.2} and Lemma \ref{Lem6.3} we derive the ``key equation'' for $ r = 2 $.
\begin{thm} \label{Thm6.4}
If $ ( 2 \tautz + 1 ) \tan \theta^{ ( 2 ) } < \Re ( s ) \tan \theta^{ ( 2 ) } - \Im ( s ) < 2 \tan \theta^{ ( 2 ) } $, $ \Re ( s ) > 2 ( 1 + \varepsilon^{ ( 2 ) } ) $ and $ \Re ( w ) > 2 $ then the following equation holds :
\begin{align*}
& - \sum_{ \Im ( \rhoco ) , \Im ( \rhoct ) < 0 } \frac{ 1 }{ ( s - \rhoco - \rhoct )^w } + \sum_{ \Im ( \rhoco ) , \Im ( \rhoct ) > 0 } \frac{ 1 }{ ( s - \rhoco - \rhoct )^w } \\
& + \sum_{ ( a , b ) \in \{ ( 1 , 2 ) , ( 2 , 1 ) \} } \left( \sum_{ \Im ( \rhoca ) > 0 } \sum_{ n = 1 }^\infty \frac{ 1 }{ \left( s - \rhoca + 2n - \frac{ 3 + \chi_b ( -1 ) }{ 2 } \right)^w } \right. \\
& + \sum_{ \Im ( \rhocb ) > 0 } \frac{ \mutauaz }{ \left( s - \rhocb - \frac{ 1 }{ 2 } - \taucaz \right)^w } + \sum_{ \Im ( \rhocb ) > 0 } \frac{ \mutauaz }{ \left( s - \rhocb - \frac{ 1 }{ 2 } + \taucaz \right)^w }  \\
& + \sum_{ \Im ( \rhocb ) > 0 } \frac{ \muaz }{ \left( s - \rhocb - \frac{ 1 }{ 2 } \right)^w } + \sum_{ n = 1 }^\infty \frac{ \mutauaz }{ \left( s + 2 n - 2 - \frac{ \chi_b ( -1 ) }{ 2 } - \taucaz \right)^w } \\
& + \left. \sum_{ n = 1 }^\infty \frac{ \mutauaz }{ \left( s + 2n - 2 - \frac{ \chi_b ( -1 ) }{ 2 } + \taucaz \right)^w } + \sum_{ n = 1 }^\infty \frac{ \muaz }{ \left( s + 2n - 2 - \frac{ \chi_b ( -1 ) }{ 2 } \right)^w } \right) \\
&+ \sum_{ n_1 = 1 }^\infty \sum_{ n_2 = 1 }^\infty \frac{ 1 }{ \left( s + 2 n_1 + 2 n_2 - 3 - \frac{ \chi_1 ( -1 ) + \chi_2 ( - 1 )  }{ 2 } \right)^w } \\
& + \frac{ \mutauoz \mutautz }{ \left( s - 1 - \taucoz - \tauctz \right)^w } + \frac{ \mutauoz \mutautz }{ \left( s - 1 - \taucoz + \tauctz \right)^w } \\
& + \frac{ \mutauoz \mutautz }{ \left( s - 1 + \taucoz - \tauctz \right)^w } + \frac{ \mutauoz \mutautz }{ \left( s - 1 + \taucoz + \tauctz \right)^w } \\
& + \sum_{ ( a , b ) \in \{ ( 1 , 2 ) , ( 2 , 1 ) \} } \left( \frac{ \mutauaz \mubz }{ \left( s - 1 - \taucaz \right)^w } + \frac{ \mutauaz \mubz }{ \left( s - 1 + \taucaz \right)^w } \right) + \frac{ \muoz \mutz }{ ( s - 1 )^w } \\
& = - \frac{ 1 }{ \Gamma ( w ) } \sum_{ k = 1 }^{ 10 } E_k ( w , s , \{ \chi_j \}_{ j = 1 }^2 ).
\end{align*}
\end{thm}
\begin{proof}[\underline{\textbf{Proof of Theorem \ref{Thm6.4}}}]
We put $ r = 2 $ and $ z = - i ( s - 1 ) $ in Theorem \ref{Thm4.1} and then by applying Lemma \ref{Lem6.1} and Lemma \ref{Lem6.2} and replacing $ \bar{ \chi } $ with $ \chi $ we obtain the desired result.
\end{proof}
\subsection{Proof of Theorem \ref{Thm1.3}}
\begin{proof}
The left-hand side of the formula in Theorem \ref{Thm6.4} is a meromorphic function of $ w $ on the whole $ \mathbb{ C } $ by Corollary \ref{Cor4.3}. Hence, by using the definition of zeta regularized products we have
\begin{align*}
& \exp \left( - \underset{ w = 0 }{ \mathrm{ Res } } \left( \frac{ \mbox{the left-hand side of the formula in Theorem } \ref{Thm6.4} }{ w^2 } \right) \right) \\
& = \Rprod_{ \Im ( \rhoco ) , \Im ( \rhoct ) < 0 } (( s - \rhoco - \rhoct ))^{ - 1 } \Rprod_{ \Im ( \rhoco ) , \Im ( \rhoct ) > 0 } (( s - \rhoco - \rhoct )) \\
& \quad \times \prod_{ ( a , b ) \in \{ ( 1 , 2 ) , ( 2 , 1 ) \} } \left( \Rprod_{ \Im ( \rhoca ) > 0 , n \ge 1 } \left( \left( s - \rhoca + 2n - \frac{ 3 + \chi_b ( - 1 ) }{ 2 } \right) \right) \right. \\
& \phantom{ \quad \quad \times \prod_{ ( a , b ) \in \{ ( 1 , 2 ) , ( 2 , 1 ) \} } } \times \Rprod_{ \Im ( \rhoca ) > 0 } \left( \left( s - \rhoca - \frac{ 1 }{ 2 } - \taucbz \right) \right)^{ \mutaubz } \\
& \phantom{ \quad \quad \times \prod_{ ( a , b ) \in \{ ( 1 , 2 ) , ( 2 , 1 ) \} } } \times \Rprod_{ \Im ( \rhoca ) > 0 } \left( \left( s - \rhoca - \frac{ 1 }{ 2 } + \taucbz \right) \right)^{ \mutaubz } \\
& \phantom{ \quad \quad \times \prod_{ ( a , b ) \in \{ ( 1 , 2 ) , ( 2 , 1 ) \} } } \times \Rprod_{ \Im ( \rhoca ) > 0 } \left( \left( s - \rhoca - \frac{ 1 }{ 2 } \right) \right)^{ \mubz } \\
& \phantom{ \quad \quad \times \prod_{ ( a , b ) \in \{ ( 1 , 2 ) , ( 2 , 1 ) \} } } \times \Rprod_{ n \ge 1 } \left( \left( s + 2n - 2 - \frac{ \chi_a ( - 1 ) }{ 2 } - \taucbz \right) \right)^{ \mutaubz } \\
& \phantom{ \quad \quad \times \prod_{ ( a , b ) \in \{ ( 1 , 2 ) , ( 2 , 1 ) \} } } \times \Rprod_{ n \ge 1 } \left( \left( s + 2n - 2 - \frac{ \chi_a ( - 1 ) }{ 2 } + \taucbz \right) \right)^{ \mutaubz } \\
& \phantom{ \quad \quad \times \prod_{ ( a , b ) \in \{ ( 1 , 2 ) , ( 2 , 1 ) \} } } \times \left. \Rprod_{ n \ge 1 } \left( \left( s + 2n - 2 - \frac{ \chi_a ( - 1 ) }{ 2 } \right) \right)^{ \mubz } \right) \\
& \quad \times \Rprod_{ n_1 , n_2 \ge 1 } \left( \left( s + 2 n_1 + 2 n_2 - 3 - \frac{ \chi_1 ( -1 ) + \chi_2 (-1) }{ 2 } \right) \right) \\
& \quad \times \left( ( s - 1 - \taucoz - \tauctz ) ( s - 1 - \taucoz + \tauctz ) \right. \\
& \quad \quad \times \left. ( s - 1 + \taucoz - \tauctz ) ( s - 1 + \taucoz + \tauctz ) \right)^{ \mutauoz \mutautz } \\
& \quad \times \prod_{ ( a , b ) \in \{ ( 1 , 2 ) , ( 2 , 1 ) \} } \left( ( s - 1 - \taucaz )( s - 1 + \taucaz ) \right)^{ \mutauaz \mubz } \\
& \quad \times ( s - 1 )^{ \muoz \mutz } \\
& = ( L_{ \chi_1 } \underset{ \mathbb{ F }_1 }{ \otimes } L_{ \chi_2 } ) ( s ).
\end{align*}

On the other hand, by Theorem \ref{Thm6.4} and noting that $ \frac{ 1 }{ \Gamma ( w ) } = w + O ( w^2 ) \ ( w \to 0 ) $, we have
\begin{align*}
& \exp \left( - \underset{ w = 0 }{ \mathrm{ Res } } \left( \frac{ \mbox{the right-hand side of the formula in Theorem } \ref{Thm6.4} }{ w^2 } \right) \right) \\
& = \exp \left( \sum_{ k = 1 }^{ 10 } E_k ( 0 , s , \{ \chi_j \}_{ j = 1 }^2 ) \right)
\end{align*}
for $ \Re ( s ) > 2 $. This completes the proof.
\end{proof}

\end{document}